\documentclass[12pt]{article}
\usepackage{amsfonts}
\usepackage{stmaryrd}
\usepackage{graphicx}
\usepackage{epsfig}
\usepackage{xcolor}
\usepackage{amsmath}
\usepackage{amsthm}
\usepackage{amssymb}
\usepackage{mathrsfs}
\usepackage{float}
\usepackage{multirow}
\usepackage{verbatim}
\usepackage{fancyhdr}
\usepackage{subfigure}
\usepackage{color}
\usepackage{mathtools}
\usepackage{mathrsfs}
\usepackage{sectsty}
\usepackage[title]{appendix}
\usepackage{threeparttable}
\usepackage{dcolumn}
\usepackage{booktabs}
\usepackage{indentfirst}
\usepackage{setspace}
\usepackage{bm}
\usepackage{enumerate}
\usepackage{lineno}
\usepackage{bbm}
\usepackage[labelfont=bf,labelsep=period]{caption}
\newcolumntype{d}[1]{D{.}{.}{#1}}
\usepackage{hyperref}
\hypersetup{colorlinks,linkcolor=blue,citecolor=blue}
\newtheorem{definition}{Definition}[section]
\newtheorem{theorem}{Theorem}[section]

\newtheorem{lemma}{Lemma}[section]

\newtheorem{remark}{Remark}[section]

\allowdisplaybreaks

\begin{document}


\pagenumbering{arabic}
\baselineskip=1.3pc

\vspace*{0.5in}

\begin{center}

{\Large{\bf Structure-preserving nodal DG method for the Euler equations with gravity: well-balanced, entropy stable, and positivity preserving}}

\end{center}

\vspace{.03in}

\centerline{
Yuchang Liu\footnote{School of Mathematical Sciences,
         University of Science and Technology of China,
         Hefei, Anhui 230026, P.R. China.  
         E-mail: lissandra@mail.ustc.edu.cn.},
Wei Guo\footnote{Department of Mathematics and Statistics, Texas Tech University, Lubbock, TX, 70409, USA. 
E-mail: weimath.guo@ttu.edu.
Research supported by Air Force Office of Scientific Research FA9550-18-1-0257.
},
Yan Jiang\footnote{School of Mathematical Sciences,
         University of Science and Technology of China, Hefei,
         Anhui 230026, P.R. China.  
         E-mail: jiangy@ustc.edu.cn.
         Research supported by NSFC grant 12271499. },
and Mengping Zhang\footnote{School of Mathematical Sciences,
         University of Science and Technology of China, Hefei,
         Anhui 230026, P.R. China.  
         E-mail: mpzhang@ustc.edu.cn.}
}

\vspace{.1in}

\noindent
{\bf Abstract: }We propose an entropy stable and positivity preserving discontinuous Galerkin (DG) scheme for the Euler equations with gravity, which is also well-balanced for hydrostatic equilibrium states. To achieve these properties, we utilize the nodal DG framework and carefully design the source term discretization using entropy conservative fluxes. Furthermore, we demonstrate that the proposed methodology is compatible with a positivity preserving scaling limiter, ensuring positivity of density and pressure under an appropriate CFL condition. To the best of our knowledge, this is the first DG scheme to simultaneously achieve these three properties with theoretical justification. Numerical examples further demonstrate its robustness and efficiency.

\vspace{.1in}


\section{Introduction}

In this paper, we focus on the simulation of the Euler equations with gravity, which have broad applications in astrophysics and atmospheric science. The development of high-order methods has enabled the computation of high-resolution solutions of partial differential equations (PDEs) with fewer meshes. 
Among them, the discontinuous Galerkin (DG) method, introduced by Reed and Hill in 1973 \cite{reed1973triangular} and later refined and extended by Cockburn and Shu in a series of seminal works \cite{cockburn1989tvb, cockburn1989tvb3, cockburn1990runge2, cockburn1991runge}, offers several distinctive advantages, such as local conservation, high-order accuracy, flexibility in handling complex geometries and general boundary conditions, as well as ease of adaptation and parallel implementation. While DG discretization is an excellent tool for solving the Euler equations with gravity,
such a system possesses fundamental physical structures that must be carefully preserved to ensure reliable and robust simulations. 
Over the past few decades, designing structure-preserving schemes, which maintain key physical properties in the discrete sense, has become a research focus and posed a significant challenge \cite{slotnick2014cfd}. 
This paper aims to develop a provably structure-preserving DG scheme for the Euler equations with gravity.

Mathematically, the Euler equations with gravity form a nonlinear balance law.
A balance law admits a special class of solutions known as \textit{equilibrium solutions} or \textit{steady-state solutions}. At the PDE level, for these solutions, the flux term is exactly balanced by the source term, ensuring that the equilibrium state remains unchanged as time progresses. In the algorithm design, the goal is to ensure that the scheme preserves this property at the discrete level, capturing the ``numerical equilibrium solutions".
Such methods, known as well-balanced (WB) schemes, offer several desired advantages. For example, even on relatively coarse meshes, WB schemes can effectively capture small perturbations added to the equilibrium state, whereas non-WB methods often fail to do so, leading to significant errors or computational instability. Many successful WB schemes have been proposed in the literature, most of which are for the shallow water equations over non-flat bottom topographies -- another important class of hyperbolic balance laws, see, e.g., \cite{xing2010positivity, xing2011advantage,bollermann2013well, michel2016well, del2024well}. WB schemes for the Euler equations with gravity include, e.g., \cite{chandrashekar2015second, grosheintz2019high, xu2024high, arun2024energy}. In particular, under the DG framework, the primary technique for ensuring the WB property is to carefully modify the discretization of the source term in such a way that the flux term and source term are exactly balanced in a discrete sense. In \cite{li2016well}, Li and Xing proposed a WB DG method for isothermal hydrostatic balance. In \cite{chandrashekar2017well}, Chandrashekar and Zenk developed a nodal DG scheme that is WB for both isothermal and isentropic hydrostatic balances, utilizing an interpolation property of the numerical solution. In \cite{wu2021uniformly}, by using the properties of the HLLC flux, Wu and Xing developed a WB DG method for arbitrary hydrostatic equilibrium. More recently, in \cite{du2024well}, Du, Yang, and Zhu proposed a DG method that is also WB for arbitrary hydrostatic equilibrium, utilizing the Lax-Friedrichs flux with an adjusted dissipation coefficient.

Moreover, for hyperbolic conservation laws, the entropy condition is crucial for well-posedness analysis, ensuring that the total entropy does not increase over time, in accordance with the second law of thermodynamics. Hence, it is desirable for the numerical solution to satisfy a discrete form of the entropy condition. Schemes that achieve this property are referred to as entropy-stable \textit{entropy stable} (ES) schemes. In recent years, the development of ES methods has attracted considerable research interest. 
Regarding the DG discretization, there exist two primary approaches for achieving entropy stability. The first approach \cite{chen2017entropy, liu2018entropy, liu2025globally} relies on the construction of summation-by-parts (SBP) operators and the application of the flux differencing technique. The second approach \cite{abgrall2018general, gaburro2023high, liu2024non} attempts to directly control entropy production by incorporating an artificial term into the DG formulation. 
A simple yet important observation made in \cite{desveaux2016well} is that the Euler equations with gravity and those without a source term share the same entropy pair and entropy condition, which is useful for developing a scheme that is both WB and ES.  In \cite{waruszewski2022entropy}, Waruszewski et al. devised an ES DG scheme for the rotating Euler equation with gravity utilizing the first approach. However, their scheme does not account for the WB property and hence cannot capture isothermal equilibrium. In this work, we use the first approach as a foundation to develop a DG scheme for the Euler equations with gravity that preserves both ES and WB properties.

Furthermore, for the Euler equations, the density and pressure of the solution must remain positive to ensure well-posedness and prevent instability or simulation failure \cite{zhang2010positivity}. A numerical method that ensures positivity is known as a positivity-preserving (PP) scheme. In the pioneering work \cite{zhang2010positivity}, Zhang and Shu proposed a simple high-order scaling-based PP DG method to solve the Euler equations.
In \cite{zhang2011positivity}, the technique is extended to the Euler equations with source terms. In the aforementioned works \cite{wu2021uniformly, du2024well}, a similar approach is also applied to achieve the PP property for the Euler equations with gravity.

To our best knowledge, no DG scheme exists that is simultaneously WB, ES, and PP. In this work, we aim to develop a high-order structure-preserving nodal DG scheme that satisfies all these properties. Our contributions encompass the following aspects. Under the ES nodal DG framework, we propose a novel approach to modify the discretization of the source term to precisely match the form of entropy conservative fluxes, thereby ensuring exact balance for the numerical equilibrium state solution while preserving the ES property. Furthermore, we prove that the scheme is WB for arbitrary hydrostatic equilibrium states and both ES and PP for general states. A comprehensive suite of numerical experiments demonstrates the crucial role of each preserved property in our formulation. The numerical evidence indicates that omitting any single property leads to either computational breakdown or violations of critical physical laws in specific test configurations, highlighting the effectiveness of the proposed structure-preserving methodology.

The remainder of this paper is organized as follows. In Section \ref{sec2}, we introduce the Euler equations with gravity, together with their equilibrium solutions and entropy condition. Section \ref{sec3} presents our structure-preserving nodal DG scheme for the one-dimensional (1D) case and its theoretical analysis.  In Section \ref{sec4}, we extend this method to two dimensions (2D). In Section \ref{sec5}, various numerical examples are provided to demonstrate the performance of our proposed scheme. Section \ref{sec6} concludes with a discussion of our results and potential directions for future work.

\section{Euler Equations with Gravity}\label{sec2}

\subsection{Governing Equations}

In the general $d$-dimensional case ($d=1,2,3$), the Euler equations with gravity can be expressed as the following nonlinear system of balance laws:
\begin{equation}\label{eq:Euler-d}
\mathbf{U}_t + \nabla \cdot \mathbf{F}\left( \mathbf{U} \right) = \mathbf{S}\left( \mathbf{U}, \mathbf{x} \right), \quad (\mathbf{x}, t) \in \mathbb{R}^d \times [0, +\infty),
\end{equation}
where
\begin{equation}
\mathbf{U} = \left[ \begin{array}{c}
    \rho \\
    \rho \mathbf{u} \\
    \mathcal{E} \\
\end{array} \right], \quad \mathbf{F}\left( \mathbf{U} \right) = \left[ \begin{array}{c}
    \rho \mathbf{u} \\
    \rho \mathbf{u} \otimes \mathbf{u} + p I_d \\
    \mathbf{u} \left( \mathcal{E} + p \right) \\
\end{array} \right], \quad \mathbf{S}\left( \mathbf{U}, \mathbf{x} \right) = \left[ \begin{array}{c}
    0 \\
    -\rho \nabla \phi \\
    -\rho \mathbf{u} \cdot \nabla \phi \\
\end{array} \right].
\end{equation}
Here, $\rho$ denotes the density, $\rho \mathbf{u}$ is the momentum, $\mathcal{E}$ represents the total energy, $I_d$ is the identity matrix, and $p$ is the pressure. The source term $\mathbf{S}$ models the gravitational force, where $\phi = \phi(\mathbf{x})$ denotes the gravitational potential and is time-independent. To close the system, the ideal gas equation of state is used:
\begin{equation}\label{eq:EOS}
\mathcal{E} = \frac{p}{\gamma - 1} + \frac{1}{2} \rho \left\|\mathbf{u}\right\|^2, \quad 
\text{with}\,\, \gamma = 1.4.
\end{equation}
 The density and pressure in the solution of \eqref{eq:Euler-d} must remain positive for the well-posedness. That is we require 
$\mathbf{U}(\mathbf{x}, t) \in \mathscr{G}, 
 \forall (\mathbf{x}, t) \in \mathbb{R}^d \times [0, +\infty),$ where  $\mathscr{G}$ denotes the \textit{admissible set}
\begin{equation}\label{eq:admissible}
\mathscr{G} = \left\{ (\rho, \rho \mathbf{u}, \mathcal{E}) : \ \rho > 0 \quad \text{and} \quad p(\rho, \rho \mathbf{u}, \mathcal{E}) > 0 \right\}.
\end{equation}
It can be verified that the set $\mathscr{G}$ is convex when $\rho > 0$.

\subsection{Steady-State Solutions}

With the gravitational potential $\phi$, the system \eqref{eq:Euler-d} admits a set of time-independent solutions, called {\em equilibrium state solutions}, which satisfy
$$ \nabla\cdot\mathbf F(\mathbf U)=\mathbf S(\mathbf U,\mathbf x). $$
For these solutions, the flux term is exactly balanced with the source term. In particular, in this work, we focus on the equilibrium state solutions with zero velocity, characterized by
\begin{equation}\label{eq:eqbm}
\rho = \rho(\mathbf{x}), \quad \mathbf{u} = \mathbf{0}, \quad \nabla p = -\rho \nabla \phi,
\end{equation}
referred to as the \textit{hydrostatic equilibrium state} or \textit{mechanical equilibrium state}. There are two important families of special steady-state solutions for \eqref{eq:eqbm}. 
The first type is called the \textit{isothermal hydrostatic equilibrium state}, in which the gas is assumed to obey the relation
$$ p = \rho R\, T $$
with a constant temperature $T$. Here, $R$ is the gas constant. Under this assumption, the solution of \eqref{eq:eqbm} is given by
\begin{equation}\label{eq:isT}
\rho = \rho_0 \exp \left( -\frac{\phi}{R\,T_0} \right), \quad \mathbf{u} = \mathbf{0}, \quad p = p_0 \exp \left( -\frac{\phi}{R\,T_0} \right),
\end{equation}
where $\rho_0, p_0, T_0$ are positive constants satisfying $p_0 = \rho_0 R T_0$. 
The second type is the \textit{isentropic hydrostatic equilibrium state}, which assumes the entropy of the gas remains constant, i.e.,
$$ p \rho^{-\gamma} = K_0. $$
Under this assumption, the solution of \eqref{eq:eqbm} is given by
\begin{equation}\label{eq:isE}
\rho = \left( \frac{\gamma - 1}{K_0 \gamma} \left( C - \phi \right) \right)^{\frac{1}{\gamma - 1}}, \quad \mathbf{u} = \mathbf{0}, \quad p = K_0 \rho^\gamma,
\end{equation}
where $C$ and $K_0$ are constants.

\subsection{Entropy Analysis}

Consider a $d$-dimensional hyperbolic conservation law
\begin{equation}\label{eq:hcl}
\mathbf{U}_t + \nabla \cdot \mathbf{F}(\mathbf{U}) = \mathbf{0}, \qquad 
\mathbf{F} = (\mathbf{F}_1, \dots, \mathbf{F}_d),
\end{equation}
and suppose the state $\mathbf{U}$ takes values in a convex set $\mathcal{D}$.
A convex function $\mathcal{U}(\mathbf{U}): \mathcal{D} \rightarrow \mathbb{R}$ is called an entropy function for system \eqref{eq:hcl} if there exists an entropy flux $\mathcal{F}(\mathbf{U}) = (\mathcal{F}_1(\mathbf{U}), \dots, \mathcal{F}_d(\mathbf{U})): \mathcal{D} \rightarrow \mathbb{R}^d$ such that
\begin{equation}\label{eq:entropy}
\mathcal{F}_i'(\mathbf{U}) = \mathcal{U}'(\mathbf{U}) \,\mathbf{F}_i'(\mathbf{U}), \quad i = 1, \dots, d.
\end{equation}
Here, $\mathcal F_i'(\mathbf U)$ and $\mathcal U'(\mathbf U)$ are viewed as row vectors. 
The entropy function $\mathcal{U}$ and the entropy flux  $\mathcal{F}$ together form an \textit{entropy pair} $(\mathcal{U}, \mathcal{F})$. 
$\mathbf{V} = \mathcal{U}'(\mathbf{U})^T$ is called the \textit{entropy variable}. If a system of conservation laws \eqref{eq:hcl} admits an entropy pair, 
then left-multiplying by $\mathbf{V}(\mathbf{U})^T$ yields an additional entropy conservation law for smooth solutions
\begin{equation}\label{eq:EC}
\mathcal{U}(\mathbf{U})_t + \nabla \cdot \mathcal{F}(\mathbf{U}) = 0.
\end{equation}
For non-smooth solutions, the above equation is replaced by an inequality in the weak sense
\begin{equation}\label{eq:ES}
\mathcal{U}(\mathbf{U})_t + \nabla \cdot \mathcal{F}(\mathbf{U}) \le 0,
\end{equation}
 known as the \textit{entropy condition}. 
In particular, a strictly convex function $\mathcal{U}$ serves as an entropy function if and only if $\partial \mathbf{U} / \partial \mathbf{V}$ is a symmetric, positive-definite matrix and $\partial \mathbf{F}_i(\mathbf{U}(\mathbf{V})) / \partial \mathbf{V}$ is symmetric for all $i$ \cite{godunov1961interesting, godlewski2013numerical}. In this case, there exist twice-differentiable scalar functions $\varphi(\mathbf{V})$ and $\psi_i(\mathbf{V})$, called the \textit{potential function} and \textit{potential flux}, with $\varphi(\mathbf{V})$ strictly convex, such that
\begin{equation}\label{eq:psi}
\mathbf{U}(\mathbf{V})^T = \frac{\partial \varphi}{\partial \mathbf{V}}, \quad \mathbf{F}_i(\mathbf{V})^T = \frac{\partial \psi_i}{\partial \mathbf{V}}.
\end{equation}

For the Euler equations without source term $\mathbf{S}(\mathbf{U}, \mathbf{x})$, Harten \cite{HARTEN1983151} showed that there exists a family of entropy pairs related to the physical specific entropy $s = \ln(p \rho^{-\gamma})$, satisfying the symmetrization condition. 
However, to symmetrize the viscous term in the compressible Navier-Stokes equations with heat conduction \cite{HUGHES1986223}, there is only one choice of the entropy pair
\begin{equation} \label{eq:entropy1}
\mathcal{U} = -\frac{\rho s}{\gamma - 1}, \quad \mathcal{F} = -\frac{\rho s }{\gamma - 1}\mathbf{u}.
\end{equation}
Correspondingly,
$$
\mathbf{V} = \mathcal{U}'(\mathbf{U})^T = \left[ \begin{array}{c}
    \dfrac{\gamma - s}{\gamma - 1} - \dfrac{\rho \|\mathbf{u}\|^2}{2p} \\
    {\rho \mathbf{u}}/{p} \\
    -{\rho}/{p} \\
\end{array} \right],
$$
and $\varphi = \rho, \  \psi_i = \rho u_i$. 
Notice that $$
\mathbf{V}(\mathbf{U})^T \mathbf{S}(\mathbf{U}, \mathbf{x}) = -\rho \nabla \phi \cdot \frac{\rho \mathbf{u}}{p} - \left( -\rho \mathbf{u} \cdot \nabla \phi \right) \frac{\rho}{p} = 0,
$$
and hence, it is straightforward to verify that the Euler equations with gravity \eqref{eq:Euler-d} also satisfy the entropy inequality \eqref{eq:ES} with the entropy pair \eqref{eq:entropy1} \cite{desveaux2016well}. This observation is critical to design a scheme that is ES and WB simultaneously. 

\section{Structure-Preserving Nodal DG Method in One Dimension}\label{sec3}
In this section, we formulate the proposed scheme for the 1D case of \eqref{eq:Euler-d}, which is given as 
\begin{equation}\label{eq:Euler1D}
\left[ \begin{array}{c}
    \rho \\
    m \\
    \mathcal{E} \\
\end{array} \right]_t + \left[ \begin{array}{c}
    m \\
    \rho u^2 + p \\
    u (\mathcal{E} + p) \\
\end{array} \right]_x = \left[ \begin{array}{c}
    0 \\
    -\rho \phi_x \\
    -m \phi_x \\
\end{array} \right],
\end{equation}
where $m = \rho u$ is the momentum. Denote such a system by
$$ \mathbf{U}_t + \mathbf{F}(\mathbf{U})_x = \mathbf{S}(\mathbf{U}, x)$$
with $\mathbf{F}=(F_1,F_2,F_3)^T$.

Assume that the 1D spatial domain $\Omega=[a,b]$ is divided into $N$ uniform cells $\mathcal K=\{K_{i}=[x_{i-1/2},x_{i+1/2}], i=1, \cdots, N\}$ with $x_{1/2}=a$, $x_{N+1/2}=b$ and mesh size $\Delta x=x_{i+1/2} - x_{i+1/2}$. 
Define the finite element space as the piecewise polynomial space of degree at most $k$:
$$
V_h^k = \left\{ w(x): w(x) |_{K_i} \in P^k(K_i), \quad \forall K_i \in \mathcal{K} \right\}.
$$
The classic semi-discrete DG scheme for \eqref{eq:Euler1D} is: Find $\mathbf{U}_h \in [V_h^k]^3$, such that for any test function $\mathbf{W} \in [V_h^k]^3$ and cell $K_i \in \mathcal{K}$,
\begin{equation}\label{eq:modal_DG}
   \int_{K_i} \frac{\partial \mathbf{U}_h}{\partial t} \cdot \mathbf{W} \, \mathrm{d}x = \int_{K_i} \mathbf{F}(\mathbf{U}_h) \cdot \frac{\partial \mathbf{W}}{\partial x} \, \mathrm{d}x - \hat{\mathbf{F}}_{i+1/2} \mathbf{W}^-_{i+1/2} + \hat{\mathbf{F}}_{i-1/2} \mathbf{W}^+_{i-1/2} + \int_{K_i} \mathbf{S} \cdot \mathbf{W} \, \mathrm{d}x.
\end{equation}
Here, $\hat{\mathbf{F}}_{i+1/2}$ represents the numerical flux at the cell interface $x_{i+1/2}$. In this work, we employ the Lax-Friedrichs flux:
\begin{equation}\label{eq:LF}
\hat{\mathbf{F}}^{LF}(\mathbf{U}_L, \mathbf{U}_R) = \frac{1}{2} (\mathbf{F}(\mathbf{U}_R) + \mathbf{F}(\mathbf{U}_L)) - \frac{\alpha}{2} (\mathbf{U}_R - \mathbf{U}_L),
\end{equation}
where $\alpha$ is an estimate of the maximum local wave speed and will be determined below. 

\subsection{Gauss-Lobatto Quadrature and the SBP Property}
We now apply the Gauss–Lobatto quadrature rule with \(k+1\) quadrature points to build the nodal DG scheme. Let 
$$
-1 = X_0 < X_1 < \cdots < X_k = 1,
$$
denote the Gauss–Lobatto quadrature points on the reference element \(I = [-1, 1]\), with corresponding quadrature weights \(\{\omega_l\}_{l=0}^{k}\). 
The difference matrix \(D\) is defined as
$
D_{jl} = L_l'(X_j),
$
where $L_l$ is the $l$-th Lagrange basis polynomial satisfying $L_l(X_j)=\delta_{lj}$. 
The mass matrix \(M\) and stiffness matrix \(S\) are given by $M = \text{diag}\{\omega_0, \omega_1, \dots, \omega_k\}$ and $S= MD$.
We recall the following properties of these matrices \cite{chen2017entropy}:

\begin{lemma}\label{lem:SBP}
Let the boundary matrix \(B\) be defined as
$$
B = \mathrm{diag}\{-1, 0, 0, \dots, 0, 0, 1\} =: \mathrm{diag}\{\tau_0, \dots, \tau_k\},
$$
then $S + S^T = B.$
\end{lemma}

\begin{lemma}
For each \(0 \le j \le k\), we have
$$
\sum_{l=0}^{k} D_{jl} = \sum_{l=0}^{k} S_{jl} = 0, \quad \sum_{l=0}^{k} S_{lj} = \tau_j.
$$
\end{lemma}

Using the matrices defined above, we can construct the original nodal DG scheme in a compact matrix-vector formulation based on the nodal values. First, we introduce the following shorthand notation:
\begin{align}
 x_i(X) = x_{i} + \frac{\Delta x}{2} X,\quad
 \mathbf{U}_{i_1}^{i} = \mathbf{U}_h(x_i(X_{i_1})), \quad \mathbf{F}_{i_1}^{i} = \mathbf{F}(\mathbf{U}_{i_1}^{i}),
\end{align}
\begin{equation}
\label{eq:Fstar}
\mathbf{F}_{i_1}^{*,i} = \begin{cases}
\hat{\mathbf{F}}_{i-1/2}=:\hat{\mathbf F}(\mathbf U_k^{i-1},\mathbf U_0^{i}), & i_1 = 0, \\
0, & 0 < i_1 < k, \\
\hat{\mathbf{F}}_{i+1/2}=:\hat{\mathbf F}(\mathbf U_k^i,\mathbf U_0^{i+1}), & i_1 = k.
\end{cases}
\end{equation}
Then the solution is represented by
\begin{align}
\mathbf{U}_h(x)|_{K_i} = \sum_{i_i=0}^{k} \mathbf{U}_{i_1}^{i} L_{i_1} \left( \frac{x-x_i}{\Delta x/2} \right).
\end{align}
Then, after applying Gauss–Lobatto quadrature and the SBP properties in Lemma \ref{lem:SBP}, \eqref{eq:modal_DG} transforms into a nodal DG scheme, with the strong formulation given by
\begin{equation}\label{eq:nodal-old}
\frac{\Delta x}{2}\frac{\mathrm{d}\mathbf{U}_{i_1}^{i}}{\mathrm{d}t} + \sum_{l=0}^k D_{i_1, l} \mathbf{F}_l^i + \frac{\tau_{i_1}}{\omega_{i_1}} \left( \mathbf{F}_{i_1}^{*,i} - \mathbf{F}_{i_1}^{i} \right) =\frac{\Delta x}{2} \mathbf{S} \left( \mathbf{U}_{i_1}^{i}, x_i(X_{i_1}) \right)
\end{equation}
for $i=1,\ldots,N$, $i_1=0,\ldots,k$.
Meanwhile, it is known that the nodal DG scheme above may fail to preserve physical structures of interest. In this work,  we systematically modify this formulation to ensure that the resulting scheme simultaneously preserves the following properties:

\begin{itemize}
\item[1.] \textbf{Well-balanced (WB) property}: 
If the initial condition is a steady-state solution $\mathbf{U}^e(x) = (\rho^e(x), 0, \mathcal{E}^e(x))^T$, the numerical solution remains \(\mathbf{U}_h = \mathbf{U}_h^e\), where $\mathbf{U}_h^e$ denotes the interpolating polynomial of $\mathbf{U}^e$ at Gauss-Lobatto points.

\item[2.] \textbf{Entropy-stable (ES) property}: 
Under the assumption of periodic, compactly supported, or reflective boundary conditions, the semi-discrete scheme satisfies
$$
\frac{\mathrm{d} }{\mathrm{d}t} \left( \sum_{i=1}^N \sum_{i_1=0}^k \frac{\Delta x}{2} \omega_{i_1} 
 \mathcal{U}(\mathbf{U}_{i_1}^{i}) \right)  \le 0.
$$

\item[3.] \textbf{Positivity-preserving (PP) property}: if the numerical solution at time $t^n$ satisfies \(\mathbf{U}_{i_1}^{i, n} \in \mathscr{G}\), then the updated solution at \(t^{n+1}\) satisfies
$$
\mathbf{U}_{i_1}^{i, n+1} \in \mathscr{G}\quad 
\text{for any \(i\) and \(i_1\)}.
$$
\end{itemize}

\subsection{Proposed scheme}\label{sec:scheme}

For a given steady state $\mathbf U^e$ that satisfies \eqref{eq:eqbm}, we have the following relations:
$$
-\rho \phi _x = -\frac{\rho}{\rho ^e} \rho ^e \phi _x = \frac{\rho}{\rho ^e} p_{x}^{e}, \quad -m \phi _x = -\frac{m}{\rho ^e} \rho ^e \phi _x = \frac{m}{\rho ^e} p_{x}^{e}.
$$
Hence, following the approach in \cite{li2016well, wu2021uniformly, du2024well}, we rewrite the momentum and energy equations as
\begin{equation}
     \frac{\partial m}{\partial t} + \frac{\partial \left( \rho u^2 + p \right)}{\partial x} = \frac{\rho}{\rho ^e} \frac{\partial p^e}{\partial x},  \quad 
\frac{\partial \mathcal{E}}{\partial t} + \frac{\partial \left( u \left( \mathcal{E} + p \right) \right)}{\partial x} = \frac{m}{\rho ^e} \frac{\partial p^e}{\partial x},
\end{equation}
which facilitates the design of  discretization for the source term that precisely matches the flux term in the equilibrium state solution.
By utilizing the entropy conservative fluxes \cite{chen2017entropy} and the fact that $p^e = F_2(\mathbf U^e)$, the proposed nodal DG scheme is given by 
\begin{equation}\label{eq:nodalES}
\frac{\Delta x}{2}\frac{\mathrm d\mathbf U_{i_1}^i}{\mathrm dt}+\sum\limits_{l=0}^k2D_{i_1,l}\mathbf F^S(\mathbf U_{i_1}^i,\mathbf U_{l}^i)+\frac{\tau _{i_1}}{\omega_{i_1}}(\mathbf F_{i_1}^{*,i}-\mathbf F_{i_1}^i)=\mathbf S_{i_1}^i,
\end{equation}
where $\mathbf F_{i_1}^{*,i}$ is given in \eqref{eq:Fstar}, and 
\begin{equation}\label{eq:wbsource}
\mathbf S_{i_1}^i=\left( 0,\ \rho_{i_1}^i\Theta_{i_1}^i,\ m_{i_1}^i\Theta_{i_1}^i \right)^T
\end{equation}
is a high-order approximation to the source term with 
\begin{equation}
    \Theta_{i_1}^i := \frac{1}{\rho_{i_1}^{e,i}} \sum\limits_{l=0}^k 2D_{i_1,l} F_2^S(\mathbf U_{i_1}^{e,i}, \mathbf U_{l}^{e,i}) = \frac{1}{\rho_{i_1}^{e,i}} (p_x^e)_{i_1}^i\frac{\Delta x}{2} + \mathcal O(\Delta x^{k+1}). 
\end{equation}
In addition, $\mathbf F^S = (F_1^S, F_2^S, F_3^S)^T$ denotes an entropy conservative flux, and $\hat{\mathbf F} = (\hat F_1, \hat F_2, \hat F_3)^T$ used in \eqref{eq:Fstar} denotes an entropy stable flux, each with specific definitions given below.

\begin{definition}
A consistent, symmetric two-point numerical flux $\mathbf F^S(\mathbf U_L, \mathbf U_R)$ is said to be entropy conservative, if for the given entropy function $\mathcal U$,
\begin{equation}\label{eq:ESflux}
\left( \mathbf{V}_R - \mathbf{V}_L \right) ^T \mathbf{F}^S \left( \mathbf{U}_L, \mathbf{U}_R \right) = \left( \psi_R - \psi_L \right).
\end{equation}
\end{definition}

\begin{definition}\label{eq:entropystableflux}
A consistent two-point numerical flux $\hat{\mathbf F}(\mathbf U_L, \mathbf U_R)$ is said to be entropy stable, if for the given entropy function $\mathcal U$,
$$
\left( \mathbf{V}_R - \mathbf{V}_L \right) ^T \hat{\mathbf{F}} \left( \mathbf{U}_L, \mathbf{U}_R \right) \le \left( \psi_R - \psi_L \right).
$$
\end{definition}

For the Euler equations, Chandrashekar \cite{chandrashekar2013kinetic} proposed the following entropy conservative flux:
$$
\begin{aligned}
F_{1}^{S} &= \hat{\rho} \bar{u}, \\
F_{2}^{S} &= \frac{\bar{\rho}}{2\bar{\beta}} + \bar{u} F_{1}^{S}, \\
F_{3}^{S} &= \left( \frac{1}{2(\gamma - 1)\hat\beta} - \frac{1}{2} \overline{u^2} \right) F_1^S + \bar{u} F^2_S,
\end{aligned}
$$
where $\beta = \rho / 2p$, and 
$$
\overline{\alpha} = \frac{\alpha_l+\alpha_r}{2}, \quad \hat{\alpha} = \frac{\alpha_r - \alpha_l}{\ln \alpha_r - \ln \alpha_l}.
$$

For entropy stable fluxes, in \cite{toro2013riemann} Toro suggests the two-rarefaction approximation, and in \cite{guermond2016fast} Guermond and Popov demonstrated that the two-rarefaction approximated wave speeds can achieve entropy stability for the Euler equations with $1 < \gamma \le 5/3$ when using the Lax-Friedrichs flux. Therefore, we set the value of $\alpha$ in the Lax-Friedrichs flux \eqref{eq:LF} as
$$
\alpha = \max \left\{ \left| u_L \right| + c_L, \left| u_R \right| + c_R, \alpha^{RRF}(\mathbf U_L, \mathbf U_R) \right\}
$$
to ensure entropy stability, where $c = \sqrt{\gamma p / \rho}$, and $\alpha^{RRF}$ is determined using the two-rarefaction approximation technique. 

The semi-disctere nodal DG scheme \eqref{eq:nodalES} can be presented as an ODE system
$$
\frac{\mathrm{d} \mathbf{U}_h}{\mathrm{d}t} = \mathcal{L}_h(\mathbf{U}_h).
$$
The fully-discrete scheme is obtained by applying a strong stability preserving Runge-Kutta (SSP-RK) method to solve this ODE system. To further ensure the PP property, we propose to incorporate the well-established PP limiter \cite{zhang2010positivity}.
Since an SSP-RK method is a convex combination of forward Euler steps, we focus on the fully-discrete scheme with forward Euler time discretization, which is given by
\begin{equation}\label{eq:EF0}
\mathbf{U}_h^{EF} = \mathbf{U}_h^n + \Delta t \cdot \mathcal{L}_h(\mathbf{U}_h^n).
\end{equation}
The high order scaling-based PP limiter, denoted by $\Pi_h$, is applied to guarantee the positivity of the density and pressure at each Gauss–Lobatto point,
\begin{equation}\label{eq:EF}
\mathbf{U}^{n+1}_h = \Pi_h \,\mathbf{U}_h^{EF}= \Pi_h \left( \mathbf{U}_h^n + \Delta t \cdot \mathcal{L}_h(\mathbf{U}_h^n) \right).
\end{equation}
In particular, $\Pi_h$ scales the numerical solutions towards the cell averages in $K_i$ using two parameters $\theta_i^{(1)}, \theta_i^{(2)} \in [0, 1]$, as defined by
\begin{equation}\label{pp}
\begin{aligned}
\tilde{\rho}^{i,EF}_{i_1} &= \left( 1 - \theta_i^{(1)} \right) \bar{\rho}^{i,EF} + \theta_i^{(1)} \rho^{i,EF}_{i_1}, \\
\left( \begin{array}{c}
    \rho_{i_1} \\
    m_{i_1} \\
    \mathcal{E}_{i_1}
\end{array} \right)^{i,n+1} &= \left( 1 - \theta_{i}^{(2)} \right) \left( \begin{array}{c}
    \bar{\rho} \\
    \bar{m} \\
    \bar{\mathcal{E}}
\end{array} \right)^{i,EF} + \theta_{i}^{(2)} \left( \begin{array}{c}
    \tilde{\rho}_{i_1} \\
    m_{i_1} \\
    \mathcal{E}_{i_1}
\end{array} \right)^{i,EF}.
\end{aligned}
\end{equation}
Here, 
 $$
 \bar{\mathbf{U}}^{i, EF} 
 = \frac{1}{\Delta x} \int_{K_i} \mathbf{U}_h^{EF}(x) \, \mathrm{d}x 
 = \sum_{i_1=0}^k \frac{\omega_{i_1}}{2} \mathbf{U}_{i_1}^{i, EF}  
 $$
is the cell average of \(\mathbf{U}_h^{EF}\) in \(K_i\).
The value of $\theta$ is computed as
\begin{equation} \label{eq:rhoPP}
\theta_i^{(1)} = \min\left\{ \frac{\bar{\rho}^{i,EF} - \varepsilon}{\bar{\rho}^{i,EF} - \rho_m}, 1 \right\}, \quad 
\rho_m = \min\limits_{0 \le i_1 \le k} \rho^{i,EF}_{i_1},
\end{equation}
\begin{equation}\label{eq:tPP}
\theta_i^{(2)} = \min\limits_{0 \le i_1 \le k}\{ t_{i_1} \}, \quad t_{i_1} = \begin{cases}
    1, & p\left( \widetilde{\mathbf{U}}_{i_1}^{i,EF} \right) \ge \varepsilon, \\
    \tilde{t}_{i_1}, & p\left( \widetilde{\mathbf{U}}_{i_1}^{i,EF} \right) < \varepsilon.
\end{cases}
\end{equation}
Here, the parameter $\varepsilon$ is a small number, and we set $\varepsilon=10^{-13}$ as suggested in \cite{zhang2010positivity}. In \eqref{eq:tPP}, $\widetilde{\mathbf U}_{i_1}^{i,EF}=(\tilde\rho_{i_1}^{i,EF},m_{i_1}^{i,EF},\mathcal E_{i_1}^{i,EF})^T$, and the scaling factor $\tilde t_{i_1}$ is solved from the nonlinear equation
$$ p\left((1-\tilde t_{i_1})\bar{\mathbf U}^{i,n} + \tilde t_{i_1}\widetilde{\mathbf U}_{i_1}^{i,n}\right) = \varepsilon.$$
In particular, if $\mathbf U_{i_1}^{i,EF}\in \mathscr{G}$ for all $i_1$, then $\theta_i^{(1)}=\theta_i^{(2)}=1$. 
It should be noted that $\Pi_h$ does not change the cell average, i.e. $\bar{\mathbf U}^{i,n+1} = \bar{\mathbf U}^{i,EF}$. 
Moreover, if $\bar{\mathbf U}^{i,EF}\in \mathscr{G}$, the limiter \eqref{pp} ensures the nodal value at each Gauss-Lobatto point $\mathbf{U}_{i_1}^{i,n+1} \in \mathscr{G}$.

\subsection{Properties of the scheme}

In this subsection, we prove that the scheme \eqref{eq:nodalES} is WB, ES, and the fully-discrete scheme \eqref{eq:EF} is PP under a suitable CFL condition. 

\subsubsection{Well-balanced property}
\begin{theorem}[Well-balancedness]\label{thm:wb1}
The scheme \eqref{eq:nodalES} is WB for a general known stationary hydrostatic solution
$\mathbf U^e=(\rho^e,0,\mathcal E^e)$, i.e.
\begin{equation}
    \mathcal L_h(\mathbf U_h)=\mathbf 0,\qquad \mathrm{if}\ \ \mathbf U_h=\mathbf U_h^e.
\end{equation}
\end{theorem}

\begin{proof}
As $\mathbf U_h=\mathbf U_h^e$, we have $\mathbf U_{i_1}^i=\mathbf U_{i_1}^{e,i}$. It is easy to verify
\begin{equation}\label{eq:wbpf1}
F_1=F_1^S=\hat F_1=0,\quad F_3=F_3^S=\hat F_3=0, 
\end{equation}
thus $\mathcal{L}_{h,1}(\mathbf{U}_h)|_{i_1}^i =  \mathcal{L}_{h,3}(\mathbf{U}_h)|_{i_1}^i = 0$. 
For the equation of momentum, we note that $\mathbf U^e_h$ is continuous at the interfaces $x_{i+1/2}$, since it interpolates the steady state solution $\mathbf U^e$ at Gauss-Lobatto points, which include the interface location.  Utilizing the consistency of $\hat{\mathbf F}$, we have
$$ F_{2,i_1}^i=p^{e,i}_{i_1},\quad \hat F_{2,i+1/2}=p^{e}(x_{i+1/2})=F_{2,k}^i=F_{2,0}^{i+1}. $$
Therefore,
\begin{align*}
\mathcal{L}_{h,2}(\mathbf{U}_h)|_{i_1}^i
&=-\sum_{l=0}^k{2D_{i_1,l}F_{2}^{S}\left( \mathbf{U}_{i_1}^{i},\mathbf{U}_{l}^{i} \right)}-\frac{\tau _{i_1}}{\omega _{i_1}}\left( F_{2,i_1}^{*,i}-F_{2,i_1}^{i} \right) 
\\&\quad\, +\frac{\rho _{i_1}^{i}}{\rho _{i_1}^{e,i}}\sum_{l=0}^k{2D_{i_1,l}F_{2}^{S}\left( \mathbf{U}_{i_1}^{e,i},\mathbf{U}_{l}^{e,i} \right)}
\\
&=-\sum_{l=0}^k{2D_{i_1,l}F_{2}^{S}\left( \mathbf{U}_{i_1}^{e,i},\mathbf{U}_{l}^{e,i} \right)}-0+1\sum_{l=0}^k{2D_{i_1,l}F_{2}^{S}\left( \mathbf{U}_{i_1}^{e,i},\mathbf{U}_{l}^{e,i} \right)}=0.
\end{align*}
Hence, $\mathcal L_h(\mathbf U_h) = \mathbf 0$.
\end{proof}

\begin{remark}
The proof of Theorem \ref{thm:wb1} shows the necessity of choosing the numerical equilibrium state as the Gauss–Lobatto interpolating polynomial. 
For alternative approximations, $\hat F_1\, \text{and}\, \hat F_3$ in \eqref{eq:wbpf1}  may be nonzero due to possible discontinuities at interfaces. In \cite{du2024well}, the authors proposed a specialized choice of $\alpha$ for the Lax-Friedrichs flux, where $\alpha$ vanishes when the solution $\mathbf U_h$ coincides with the numerical equilibrium state $\mathbf U_h^e$, ensuring $\hat F_1=\hat F_3=0$.
However, the modified Lax-Friedrichs flux is not entropy stable according to the definition \ref{eq:entropystableflux}, which requires $\alpha$ to be sufficiently large \cite{guermond2016fast, chen2017entropy}. 
\end{remark}

\begin{remark}
For the fully-discrete scheme \eqref{eq:EF}, if $\mathbf U_h^n=\mathbf U_h^e$, then $\mathcal L_h(\mathbf U_h^n)=\mathbf 0$, yielding $\mathbf U_h^{EF}=\mathbf U_h^e$. 
When $\mathbf U^e(x)\in\mathscr G$, we have 
$$\mathbf U_{i_1}^{EF,i} =\mathbf U^e(x_i(X_{i_1})) \in\mathscr G$$
for all $i,i_1$, resulting in $\theta_i^{(1)}=\theta_i^{(2)}=1$. Therefore, we have $\mathbf U_h^{n+1}=\Pi_h(\mathbf U_h^{EF})=\mathbf U_h^e$, i.e., the fully-discrete scheme is WB.
\end{remark}

\subsubsection{Entropy-stable property}

\begin{theorem}[Entropy stability]
Assume that the boundary is periodic, compact support, or reflective. 
Then, the semi-discrete scheme \eqref{eq:nodalES} is entropy conservative within each cell in the sense of
\begin{equation}\label{eq:entropy_con}
\frac{\mathrm{d}}{\mathrm{d}t} \left( \sum_{i_1=0}^k \frac{\Delta x}{2} \omega _{i_1} \mathcal{U} \left( \mathbf{U}_{i_1}^{i} \right) \right) = \mathcal{F} _{0}^{*,i} -\mathcal{F} _{k}^{*,i},
\end{equation}
and ES in the sense of
\begin{equation}\label{eq:entropy_sta}
\frac{\mathrm d}{\mathrm dt}\mathcal U^{tot}(\mathbf U_h)
=\frac{\mathrm{d}}{\mathrm{d}t} \left( \sum_{i=1}^N \sum_{i_1=0}^k \frac{\Delta x}{2} \omega_{i_1} \mathcal{U} \left( \mathbf{U}_{i_1}^{i} \right) \right) \le 0,   
\end{equation}
where
$$
\mathcal{F} _{k}^{*,i}=\left( \mathbf{V}_{k}^{i} \right)^T\hat{\mathbf{F}}_{i+1/2}-\psi _{k}^{i} ,\quad \mathcal{F} _{0}^{*,i}=\left( \mathbf{V}_{0}^{i} \right) ^T\hat{\mathbf{F}}_{i-1/2}-\psi _{0}^{i}.
$$
\end{theorem}

\begin{proof}
Left multiplying \eqref{eq:nodalES} by $\mathbf V^T$ and integral in cell $K_i$ yields
$$\begin{aligned}
&\sum_{i_1=0}^k{\frac{\omega _{i_1}\Delta x}{2}\frac{\mathrm{d}\mathcal{U} \left( \mathbf{U}_{i_1}^{i} \right)}{\mathrm{d}t}}=\sum_{i_1=0}^k{\frac{\omega _{i_1}\Delta x}{2}\left( \mathbf{V}_{i_1}^{i} \right) ^T\frac{\mathrm{d}\mathbf{U}_{i_1}^{i}}{\mathrm{d}t}}
\\
&=-\sum_{i_1=0}^k{\sum_{l=0}^k{2S_{i_1,l}\left( \mathbf{V}_{i_1}^{i} \right) ^T\mathbf{F}^S\left( \mathbf{U}_{i_1}^{i},\mathbf{U}_{l}^{i} \right)}}-\sum_{i_1=0}^k{\tau _{i_1}\left( \mathbf{V}_{i_1}^{i} \right) ^T\left( \mathbf{F}_{i_1}^{*,i}-\mathbf{F}_{i_1}^{i} \right)}
\\&\quad +\sum_{i_1=0}^k{\frac{\omega _{i_1}}{2}\left( \mathbf{V}_{i_1}^{i} \right) ^T\mathbf{S}_{i_1}^{i}}
\\&=:-T_1-T_2+T_3.\end{aligned}
$$
For $T_1$, \cite{chen2017entropy} has proved that
\begin{equation*}\label{eq:T1}
\sum_{i_1=0}^k{\sum_{l=0}^k{2S_{i_1,l}\left( \mathbf{V}_{i_1}^{i} \right) ^T\mathbf{F}^S\left( \mathbf{U}_{i_1}^{i},\mathbf{U}_{l}^{i} \right)}}=\sum_{i_1=0}^k{\tau _{i_1}\left( \left( \mathbf{V}_{i_1}^{i} \right) ^T\mathbf{F}_{i_1}^{i}-\psi _{i_1}^{i} \right)}.
\end{equation*}
Hence 
$
-T_1-T_2=-(\mathcal{F} _{k}^{*,i}-\mathcal{F} _{0}^{*,i}).
$  
On the other hand, we have
$$
\left( \mathbf{V}_{i_1}^{i} \right) ^T\mathbf{S}_{i_1}^{i}=\left( \frac{m_{i_1}^{i}}{p_{i_1}^{i}} \right) \left( \rho _{i_1}^{i}\Theta _{i_1}^{i} \right) -\left( \frac{\rho _{i_1}^{i}}{p_{i_1}^{i}} \right) \left( m_{i_1}^{i}\Theta _{i_1}^{i} \right) =0,
$$
which implies $T_3 = 0$. Thus, we have proved \eqref{eq:entropy_con}.

Moreover, for an entropy stable flux $\hat{\mathbf{F}}$, \cite{chen2017entropy} demonstrated that the following inequality holds at each cell interface:
\begin{equation*}\label{eq:FhatES}
-(\mathcal{F} _{k}^{*,i}-\mathcal{F} _{0}^{*,i+1})=\left( \psi _{k}^{i}-\psi _{0}^{i+1} \right) -\left( \mathbf{V}_{k}^{i}-\mathbf{V}_{0}^{i+1} \right) ^T\hat{\mathbf{F}}_{i+1/2}\le 0.
\end{equation*} 
Hence, summing \eqref{eq:entropy_con} over $i$ yields \eqref{eq:entropy_sta}.
\end{proof}

\begin{remark}
Most existing WB schemes \cite{li2016well, chandrashekar2017well, du2024well} only modify the discretization of the second component of source term $\mathbf S$, because the third component is automatically zero for a steady state solution. But in our proposed scheme, to further satisfy the entropy condition, it is also necessary to modify the discretization of the third component of $\mathbf S$ to ensure $T_3 = 0$.
\end{remark}

\begin{remark}
It has been proved in \cite{chen2017entropy} that the PP limiter $\Pi_h$ \eqref{pp} will not increase the ``total entropy" in each cell, i.e.,
\begin{equation}\label{eq:ESlimiter}
\sum\limits_{i_1=0}^k \omega_{i_1} \mathcal U((\Pi_h \mathbf U_h)_{i_1}^{i})\le 
\sum\limits_{i_1=0}^k \omega_{i_1} \mathcal U(\mathbf U_{i_1}^i).
\end{equation}
Hence, if a fully discrete scheme of \eqref{eq:nodalES} is ES, then the scheme incorporating the PP limiter remains ES as well.

\end{remark}

\subsubsection{Positivity-preserving property}

\begin{theorem}[Positivity-preserving]
Assume $\mathbf U_{i_1}^{i,n}\in\mathscr{G}$ for all $i, i_1$. Then, the solution of the fully-discrete scheme \eqref{eq:EF} satisfies $\bar{\mathbf U}^{i,n+1}\in\mathscr G$ under the time step restriction
\begin{equation}\label{eq:CFLPP}
\lambda=\frac{\Delta t}{\Delta x} <\min \left\{ \frac{\omega _0}{4\alpha _0},\min_{i,i_1} \left\{ \frac{1}{4\left|\Theta _{i_1}^{i}\right|}\sqrt{\frac{1}{\left( \gamma -1 \right) \beta _{i_1}^{i}}} \right\} \right\},\quad \alpha_0=\max\limits_{\Omega}\left\{ \left|u\right|+c \right\}. 
\end{equation}
\end{theorem}

\begin{proof}
Utilizing the definition \eqref{pp}, the PP limiter $\Pi_h$ does not change the cell average, i.e. $\bar{\mathbf U}^{i,n+1} = \bar{\mathbf U}^{i,EF}$.  
Thus, from \eqref{eq:EF}, we can obtain the equation for cell averages 
$$
\Delta x\frac{\bar{\mathbf{U}}^{i,n+1}-\bar{\mathbf{U}}^{i,n}}{\Delta t}=-\hat{\mathbf{F}}_{i+1/2}+\hat{\mathbf{F}}_{i-1/2}+\sum_{i_1=0}^k{\omega _{i_1}\mathbf{S}_{i_1}^{i}}.
$$
Following the approach in \cite{zhang2011positivity}, we derive the following:
\begin{align*}
\bar{\mathbf{U}}^{i,n+1}&=\sum_{i_1=0}^k{\frac{\omega _{i_1}}{2}\mathbf{U}_{i_1}^{i,n}}-\lambda \left[ \hat{\mathbf{F}}\left( \mathbf{U}_{k}^{i,n},\mathbf{U}_{0}^{i+1,n} \right) -\hat{\mathbf{F}}\left( \mathbf{U}_{k}^{i-1,n},\mathbf{U}_{0}^{i,n} \right) \right] +\lambda \sum_{i_1=0}^k{\omega _{i_1}\mathbf{S}_{i_1}^{i}}
\\
&=\frac{\omega _0}{4}\mathbf{U}_{0}^{i,n}-\lambda \left[ \hat{\mathbf{F}}\left( \mathbf{U}_{0}^{i,n},\mathbf{U}_{k}^{i,n} \right) -\hat{\mathbf{F}}\left( \mathbf{U}_{k}^{i-1,n},\mathbf{U}_{0}^{i,n} \right) \right] 
\\
&\quad + \frac{\omega _k}{4} \mathbf{U}_{k}^{i,n}-\lambda \left[ \hat{\mathbf{F}}\left( \mathbf{U}_{k}^{i,n},\mathbf{U}_{0}^{i+1,n} \right) -\hat{\mathbf{F}}\left( \mathbf{U}_{0}^{i,n},\mathbf{U}_{k}^{i,n} \right) \right] 
\\
&\quad +\frac{1}{2}\sum_{i_1=1}^{k-1}{\frac{\omega _{i_1}}{2}\mathbf{U}_{i_1}^{i,n}}
+\frac{1}{2}\sum_{i_1=0}^k{\frac{\omega _{i_1}}{2}\mathbf{U}_{i_1}^{i,n}}
+\lambda \sum_{i_1=0}^k{\omega _{i_1}\mathbf{S}_{i_1}^{i}}
\\
&=\frac{\omega_0}{4} \mathbf{H}_0 +\frac{\omega_k}{4}\mathbf{H}_k
+\sum_{i_1=1}^{k-1}{\frac{\omega _{i_1}}{4} \mathbf{U}_{i_1}^{i,n}}
+\sum_{i_1=0}^k \frac{\omega _{i_1}}{4} \left( \mathbf{U}_{i_1}^{i,n}+4\lambda \mathbf{S}_{i_1}^{i} \right) ,
\end{align*}
which is a convex combination of $\mathbf{H}_0$, $\mathbf{H}_k$, $\{\mathbf{U}_{i_1}^{i,n}\}_{i_1=1}^{k-1}$, $\{\mathbf{U}_{i_1}^{i,n}+4\lambda \mathbf{S}_{i_1}^{i} \}_{i_1=0}^{k}$ as $\displaystyle\sum\limits_{i_1=0}^k\omega_{i_1}=2$. 
Here
$$
\mathbf{H}_0= \mathbf{U}_{0}^{i,n}-\frac{4\lambda}{\omega _0}\left[ \hat{\mathbf{F}}\left( \mathbf{U}_{0}^{i,n},\mathbf{U}_{k}^{i,n} \right) -\hat{\mathbf{F}}\left( \mathbf{U}_{k}^{i-1,n},\mathbf{U}_{0}^{i,n} \right) \right], 
$$
$$
\mathbf{H}_k= \mathbf{U}_{k}^{i,n}-\frac{4\lambda}{\omega _k}\left[ \hat{\mathbf{F}}\left( \mathbf{U}_{k}^{i,n},\mathbf{U}_{0}^{i+1,n} \right) -\hat{\mathbf{F}}\left( \mathbf{U}_{0}^{i,n},\mathbf{U}_{k}^{i,n} \right) \right].
$$
Given the convexity of $\mathscr G$, if we have 
\begin{equation}\label{eq:ppreq}\mathbf H_0\in\mathscr G,\quad \mathbf H_k\in\mathscr G,\quad 
\mathbf U_{i_1}^{i,n}+4\lambda\mathbf S_{i_1}^i \in\mathscr G,
\end{equation}
then $\bar{\mathbf U}^{i,n+1}\in\mathscr G$. 
In \cite{zhang2011positivity}, it is proved that $\mathbf H_0, \mathbf H_k\in\mathscr G$ when $4\lambda\alpha_0/\omega_0\le 1.$ On the other hand, given \eqref{eq:wbsource}, we have
$$
\mathbf{U}_{i_1}^{i,n}+4\lambda \mathbf{S}_{i_1}^{i}=\left[ \begin{array}{c}
	\rho _{i_1}^{i,n}\\
	m_{i_1}^{i,n}+4\lambda \rho _{i_1}^{i,n}\Theta _{i_1}^{i}\\
	\mathcal{E} _{i_1}^{i,n}+4\lambda m_{i_1}^{i,n}\Theta _{i_1}^{i}\\
\end{array} \right]=:\mathbf U^\star. 
$$
Evidently, the density of state $\mathbf U^\star$ is $\rho _{i_1}^{i,n}$ and is therefore positive. Then, $\mathbf U^\star\in\mathscr G$ if and only if
\begin{equation}\label{eq:ppeq}
\mathcal{E} _{i_1}^{i,n}+4\lambda m_{i_1}^{i,n}\Theta _{i_1}^{i}-\frac{\left( m_{i_1}^{i,n}+4\lambda \rho _{i_1}^{i,n}\Theta _{i_1}^{i} \right) ^2}{2\rho _{i_1}^{i,n}}>0.
\end{equation}
Solving the inequality \eqref{eq:ppeq} yields
$$
\lambda <\frac{1}{4\left|\Theta_{i_1}^i\right|}\sqrt{\frac{1}{\left( \gamma -1 \right) \beta_{i_1}^i}},
$$
which is the condition for $\lambda$ in \eqref{eq:CFLPP}. Thus, the proof is complete.

\end{proof}

\begin{remark}
   Under the restriction \eqref{eq:CFLPP}, we apply the limiter $\Pi_h$ after each forward Euler step, which will enforce the nodal values satisfying ${\mathbf U}^{i,n+1}_{i_1}\in\mathscr G$  at Gauss-Lobatto points. 
\end{remark}

\section{Structure-preserving nodal DG  method in two dimensions}\label{sec4}

The 2D Euler equations with gravity are expressed as
\begin{equation}\label{eq:EulerG-2D}
\left[ \begin{array}{c}
	\rho\\
	m\\
	n\\
	\mathcal{E}\\
\end{array} \right] _t+\left[ \begin{array}{c}
	m\\
	\rho u^2+p\\
	\rho uv\\
	u\left( \mathcal{E} +p \right)\\
\end{array} \right] _x+\left[ \begin{array}{c}
	n\\
	\rho uv\\
	\rho v^2+p\\
	v\left( \mathcal{E} +p \right)\\
\end{array} \right] _y=\left[ \begin{array}{c}
	0\\
	-\rho \phi _x\\
	-\rho \phi _y\\
	-m\phi _x-n\phi _y\\
\end{array} \right],
\end{equation}
with $m=\rho u,\ n=\rho v$, and are denoted by
$$ \mathbf U_t+\mathbf F(\mathbf U)_x+\mathbf G(\mathbf U)_y=\mathbf S(\mathbf U,x,y). $$
The corresponding entropy fluxes $\mathcal F_1,\ \mathcal F_2$ in \eqref{eq:entropy} are denoted by $\mathcal F,\ \mathcal G$, and the potential fluxes $\psi_1,\ \psi_2$ in \eqref{eq:psi} are denoted by $\psi_F,\ \psi_G$. 
Assume that the 2D spatial domain $\Omega=[a,b]\times[c,d]$ is divided into an $N_x\times N_y$ uniform
rectangular mesh with cells $\mathcal K=\{K_{ij}=[x_{i-1/2},x_{i+1/2}]\times [y_{j-1/2},y_{j+1/2}]\}$. The mesh sizes in $x$ and $y$ directions are denoted by $\Delta x$ and $\Delta y$, respectively, where $\Delta x=x_{i+1/2}-x_{i-1/2}$ and $\Delta y =y_{j+1/2}-y_{j-1/2} $.

\subsection{Proposed scheme}

For a given stationary steady state $\mathbf U^e(x)$ of \eqref{eq:EulerG-2D}, we have
$$
-\rho \phi _x=-\frac{\rho}{\rho ^e}\rho ^e\phi _x=\frac{\rho}{\rho ^e}p_{x}^{e},\ \  
-m\phi _x=-\frac{m}{\rho ^e}\rho ^e\phi _x=\frac{m}{\rho ^e}p_{x}^{e},\ \  
-n\phi_y=-\frac{n}{\rho^e}\rho^e\phi_y=\frac{n}{\rho^e}p_y^e,
$$
and $p^e = F_2(\mathbf U^e)=G_3(\mathbf U^e)$. 
We introduce the following shorthand notation
$$
x_i\left( X \right) = x_{i} +\frac{\Delta x}{2}X,\quad  
y_j\left( Y \right) = y_{j} +\frac{\Delta y}{2}Y, $$
$$	\mathbf{U}_{i_1,j_1}^{i,j}=\mathbf{U}_h\left( x_i\left( X_{i_1} \right) ,y_j\left( Y_{j_1} \right) \right) ,\,\, 
\mathbf{F}_{i_1,j_1}^{i,j}=\mathbf{F}\left( \mathbf{U}_{i_1,j_1}^{i,j} \right) ,\,\, 
\mathbf{G}_{i_1,j_1}^{i,j}=\mathbf{G}\left( \mathbf{U}_{i_1,j_1}^{i,j} \right), 
$$
$$
    \mathbf{F}_{i_1,j_1}^{*,i,j}=\left\{ \begin{array}{ll}
    \mathbf{\hat{F}}\left( \mathbf U^{i-1,j}_{k,j_1},\mathbf U^{i,j}_{0,j_1} \right)=:\hat{\mathbf F}^{i-1/2,j}_{j_1} ,& i_1=0,\\	
    \mathbf{0},& 0<i_1<k,\\
    \mathbf{\hat{F}}\left( \mathbf U^{i,j}_{k,j_1},\mathbf U^{i+1,j}_{0,j_1} \right)=:\hat{\mathbf F}^{i+1/2,j}_{j_1} ,& i_1=k,\\	
\end{array}\right.
$$
$$
    \mathbf{G}_{i_1,j_1}^{*,i,j}=\left\{ \begin{array}{ll}
	\mathbf{\hat{G}}\left( \mathbf U^{i,j-1}_{i_1,k},\mathbf U^{i,j}_{i_1,0} \right)=:\hat{\mathbf G}^{i,j-1/2}_{i_1} ,& j_1=0,\\	
        \mathbf{0},& 0<j_1<k,\\
	\mathbf{\hat{G}}\left( \mathbf U^{i,j}_{i_1,k},\mathbf U^{i,j+1}_{i_1,0} \right)=:\hat{\mathbf G}^{i,j+1/2}_{i_1} ,& j_1=k.\\	
	\end{array}\right.
	$$
Similar to the 1D case, the proposed 2D nodal DG scheme in each cell is designed as 
\begin{equation}\label{eq:nodalES2D}
\begin{aligned}
    \frac{\mathrm{d}\mathbf{U}_{i_1,j_1}^{i,j}}{\mathrm{d}t}=&-\frac{2}{\Delta x}\sum_{l=0}^{k}{2D_{i_1,l}}\mathbf{F}^S\left( \mathbf{U}_{i_1,j_1}^{i,j},\mathbf{U}_{l,j_1}^{i,j} \right) +\frac{2}{\Delta x}\frac{\tau _{i_1}}{\omega _{i_1}}\left( \mathbf{F}_{i_1,j_1}^{i,j}-\mathbf{F}_{i_1,j_1}^{*,i,j} \right) 
	\\
	&-\frac{2}{\Delta y}\sum_{l=0}^{k}{2D_{j_1,l}}\mathbf{G}^S\left( \mathbf{U}_{i_1,j_1}^{i,j},\mathbf{U}_{i_1,l}^{i,j} \right) +\frac{2}{\Delta y}\frac{\tau _{j_1}}{\omega _{j_1}}\left( \mathbf{G}_{i_1,j_1}^{i,j}-\mathbf{G}_{i_1,j_1}^{*,i,j} \right) + \mathbf S_{i_1,j_1}^{i,j}, 
\end{aligned}
\end{equation}
    where, $\mathbf F^S$ and  $\mathbf G^S$ are the entropy conservative fluxes, $\hat{\mathbf F}$ and $\hat{\mathbf G}$ are the entropy stable fluxes, 
    $$ \mathbf S_{i_1,j_1}^{i,j}=\left( 0,\frac{2}{\Delta x}\rho_{i_1,j_1}^{i,j}\Theta_{i_1,j_1}^{i,j},\frac{2}{\Delta y}\rho_{i_1,j_1}^{i,j}\Xi_{i_1,j_1}^{i,j}, \frac{2}{\Delta x}m_{i_1,j_1}^{i,j}\Theta_{i_1,j_1}^{i,j}+\frac{2}{\Delta y}n_{i_1,j_1}^{i,j}\Xi_{i_1,j_1}^{i,j}\right)^T, $$
    and
    $$\begin{aligned}
\Theta _{i_1,j_1}^{i,j}&=\frac{1}{\rho _{i_1,j_1}^{e,i,j}}\sum_{l=0}^k{2D_{i_1,l}F^S_2\left( \mathbf{U}_{i_1,j_1}^{e,i,j},\mathbf{U}_{l,j_1}^{e,i,j} \right)},
\\ \Xi _{i_1,j_1}^{i,j}&=\frac{1}{\rho _{i_1,j_1}^{e,i,j}}\sum_{l=0}^k{2D_{j_1,l}G_3^S\left( \mathbf{U}_{i_1,l}^{e,i,j},\mathbf{U}_{i_1,j_1}^{e,i,j} \right)}.\end{aligned}
$$

Same to the 1D case, coupling with the Euler forward time discretization, the fully-discrete scheme is given by
\begin{equation}\label{eq:EF2D} 
\mathbf U^{n+1}_h = \Pi_h^{2D}(\mathbf U_h^n+\Delta t\cdot\mathcal L_h(\mathbf U_h^n))=:\Pi_h^{2D}(\mathbf U_h^{EF}),
\end{equation}
where $\Pi_h^{2D}$ is the two-dimensional scaling PP limiter to correct the density and pressure at each Gauss-Lobatto point to be positive. The definition of $\Pi^{2D}_h$ is similar to $\Pi_h$ in 1D \eqref{pp}. The only difference is that the stencil points on $K_{ij}$ are $\{(X_i(x_{i_1}),Y_j(y_{j_1}))\}_{i_1,j_1=0}^{k}$, which is the tensor product of 1D stencil points along $x$- and $y$- directions. 
Note that this stencil is different from that used in \cite{zhang2010positivity, zhang2011positivity}, which uses the tensor product of Gauss-Lobatto points and Gauss points. Here, we directly use the tensor product of Gauss–Lobatto points, which is a more natural and efficient choice within the nodal DG framework.

\subsection{Properties of the scheme}

In this section, we introduce the properties of the 2D scheme. 

\subsubsection{Well-balanced property}
\begin{theorem}[Well-balancedness]\label{thm:wb2}
The scheme \eqref{eq:nodalES2D} is well-balanced for a general known hydrostatic state solution $\mathbf U^e=(\rho^e,0,0,\mathcal E^e)$, i.e.
\begin{equation}
\mathcal L_h(\mathbf U_h)=\mathbf 0,\qquad \mathrm{if}\ \ \mathbf U_h=\mathbf U_h^e.
\end{equation}
\end{theorem}

The proof follows the same procedure as the 1D case and is therefore omitted for brevity.

\subsubsection{Entropy-stable property}
\begin{theorem}[Entropy-stability]
Assume the boundaries are periodic or compactly supported, then \eqref{eq:nodalES2D} is entropy conservative within a single element
    \begin{equation} \label{eq:entropy_conservative}
        \begin{aligned}
		&\frac{\mathrm{d}}{\mathrm{d}t}\left( \frac{\Delta x\Delta y}{4}\sum_{i_1,j_1=0}^{k}{\omega _{i_1}\omega _{j_1}\mathcal U\left(\mathbf U_{i_1,j_1}^{i,j} \right)} \right) \\&\quad 
        =-\frac{\Delta y}{2}\sum_{j_1=0}^{k}{\omega _{j_1}\left( \mathcal{F} _{k,j_1}^{*,i,j}-\mathcal{F} _{0,j_1}^{*,i,j} \right)}
        -\frac{\Delta x}{2}\sum_{i_1=0}^{k}{\omega _{i_1}\left( \mathcal{G} _{i_1,k}^{*,i,j}-\mathcal{G} _{i_1,0}^{*,i,j} \right)},
        \end{aligned}
    \end{equation}
    and ES in the sense of  \begin{equation}\label{eq:entropy_stable}
		 \frac{\mathrm{d}}{\mathrm{d}t}\left( \frac{\Delta x\Delta y}{4}\sum_{i,j=1}^{N_x,N_y}{\sum_{i_1,j_1=0}^{k}{\omega _{i_1}\omega _{j_1}\mathcal U\left(\mathbf U_{i_1,j_1}^{i,j}\right)}} \right) \le 0,
    \end{equation}where\begin{align*}
		\mathcal{F} _{k,j_1}^{*,i,j}&= \left( \mathbf{V}_{k,j_1}^{i,j} \right) ^T\mathbf{F}_{k,j_1}^{*,i,j}-\psi _{F,k,j_1}^{i,j} ,\quad 
		\mathcal{F} _{0,j_1}^{*,i,j}=\left( \mathbf{V}_{0,j_1}^{i,j} \right) ^T\mathbf{F}_{0,j_1}^{*,i,j} -\psi _{F,0,j_1}^{i,j},
		\\
		\mathcal{G} _{i_1,k}^{*,i,j}&= \left( \mathbf{V}_{i_1,k}^{i,j} \right) ^T\mathbf{G}_{i_1,k}^{*,i,j} -\psi _{G,i_1,k}^{i,j},\quad \mathcal{G} _{i_1,0}^{*,i,j}=\left( \mathbf{V}_{i_1,0}^{i,j} \right) ^T\mathbf{G}_{i_1,0}^{*,i,j}- \psi _{G,i_1,0}^{i,j} .
    \end{align*}
\end{theorem}

We omit the proof here as it is similar to the 1D case.

\subsubsection{Positivity-preserving property}

\begin{theorem}[Positivity-preserving]
Given $\mathbf U_{i_1,j_1}^{i,j,n}\in\mathscr G$ on all Gauss-Lobatto quadrature points, the fully-discrete 2D scheme \eqref{eq:EF2D} satisfies 
$\bar{\mathbf U}^{i,j,n+1}\in\mathscr G$
under the time step restriction
\begin{equation}\label{eq:CFLPP2D}
\begin{aligned}
\frac{\Delta t}{\Delta x} &<\min \left\{ \frac{\omega _0}{8\alpha _{x,0}},\min \left\{ \frac{1}{4\left|\Theta_{i_1,j_1}^{i,j}\right| }\sqrt{\frac{1}{2\left( \gamma -1 \right) \beta_{i_1,j_1}^{i,j} }} \right\} \right\}, \quad \alpha_{x,0}=\max\limits_{\Omega}\left\{ \left|u\right|+c \right\},
\\ \frac{\Delta t}{\Delta y} &<\min \left\{ \frac{\omega _0}{8\alpha _{y,0}},\min \left\{ \frac{1}{4\left|\Xi_{i_1,j_1}^{i,j} \right|}\sqrt{\frac{1}{2\left( \gamma -1 \right) \beta_{i_1,j_1}^{i,j} }} \right\} \right\}, \quad \alpha_{y,0}=\max\limits_\Omega\{\left|v\right|+c\}.
\end{aligned}
\end{equation}
\end{theorem}

\begin{proof} 
From \eqref{eq:EF2D}, we derive the equation of cell averages:
$$\begin{aligned}
\bar{\mathbf{U}}^{i,j,n+1}=&\ \bar{\mathbf{U}}^{i,j,n}-\lambda _x\sum_{j_1=0}^k{\frac{\omega _{j_1}}{2}\left( \hat{\mathbf{F}}_{j_1}^{i+1/2,j}-\hat{\mathbf{F}}_{j_1}^{i-1/2,j} \right)}
\\
&-\lambda _y\sum_{i_1=0}^k{\frac{\omega _{i_1}}{2}\left( \hat{\mathbf{G}}_{i_1}^{i,j+1/2}-\hat{\mathbf{G}}_{i_1}^{i,j-1/2} \right)}+\Delta t\sum_{i_1,j_1=0}^k{\frac{\omega _{i_1}\omega _{j_1}}{4}\mathbf{S}_{i_1,j_1}^{i,j}},
\end{aligned}
$$
where $\lambda_x=\Delta t/\Delta x,\ \lambda_y=\Delta t/\Delta y$. Since $\displaystyle\sum\limits_{i_1,j_1=0}^k \omega_{i_1}\omega_{j_1}=4$, the above equation can be rewritten in the form of a convex combination
$$\begin{aligned}
\bar{\mathbf{U}}^{i,j,n+1}=&\ \sum_{i_1=1}^{k-1}{\sum_{j_1=0}^k{\frac{\omega _{i_1}\omega _{j_1}}{16}\mathbf{U}_{i_1,j_1}^{i,j,n}}}+\sum_{i_1=0}^k{\sum_{j_1=1}^{k-1}{\frac{\omega _{i_1}\omega _{j_1}}{16}\mathbf{U}_{i_1,j_1}^{i,j,n}}}
\\
&\ +\sum_{j_1=0}^k{\frac{\omega _{i_1}\omega _{j_1}}{16}}\left( \mathbf{H}_{j_1}^{x,k}+\mathbf{H}_{j_1}^{x,0} \right) +\sum_{j_1=0}^k{\frac{\omega _{i_1}\omega _{j_1}}{16}}\left( \mathbf{H}_{i_1}^{y,k}+\mathbf{H}_{i_1}^{y,0} \right) 
\\
&\ +\sum_{i_1,j_1=0}^k{\frac{\omega _{i_1}\omega _{j_1}}{8}\left( \mathbf{U}_{i_1,j_1}^{i,j,n}+2\Delta t\cdot \mathbf{S}_{i_1,j_1}^{i,j} \right)},
\end{aligned}
$$
where
$$\begin{aligned}
\mathbf{H}_{j_1}^{x,k}=&\ \mathbf{U}_{k,j_1}^{i,j,n}-\frac{8\lambda _x}{\omega _0}\left[ \hat{\mathbf{F}}\left( \mathbf{U}_{k,j_1}^{i,j,n},\mathbf{U}_{0,j_1}^{i+1,j,n} \right) -\hat{\mathbf{F}}\left( \mathbf{U}_{0,j_1}^{i,j,n},\mathbf{U}_{k,j_1}^{i,j,n} \right) \right], 
\\
\mathbf{H}_{j_1}^{x,0}=&\ \mathbf{U}_{0,j_1}^{i,j,n}-\frac{8\lambda _x}{\omega _0}\left[ \hat{\mathbf{F}}\left( \mathbf{U}_{0,j_1}^{i,j,n},\mathbf{U}_{k,j_1}^{i,j,n} \right) -\hat{\mathbf{F}}\left( \mathbf{U}_{k,j_1}^{i-1,j,n},\mathbf{U}_{0,j_1}^{i,j,n} \right) \right], 
\\
\mathbf{H}_{i_1}^{y,k}=&\ \mathbf{U}_{i_1,k}^{i,j,n}-\frac{8\lambda _y}{\omega _0}\left[ \hat{\mathbf{G}}\left( \mathbf{U}_{i_1,k}^{i,j,n},\mathbf{U}_{i_1,0}^{i,j+1,n} \right) -\hat{\mathbf{G}}\left( \mathbf{U}_{i_1,0}^{i,j,n},\mathbf{U}_{i_1,k}^{i,j,n} \right) \right], 
\\
\mathbf{H}_{i_1}^{y,0}=&\ \mathbf{U}_{i_1,0}^{i,j,n}-\frac{8\lambda _y}{\omega _0}\left[ \hat{\mathbf{G}}\left( \mathbf{U}_{i_1,0}^{i,j,n},\mathbf{U}_{i_1,k}^{i,j,n} \right) -\hat{\mathbf{G}}\left( \mathbf{U}_{i_1,k}^{i,j-1,n},\mathbf{U}_{i_1,0}^{i,j,n} \right) \right]. 
\end{aligned}
$$
Hence, if we have $\mathbf H\in\mathscr G$ for all $\mathbf H$, and $\mathbf U_{i_1,j_1}^{i,j,n}+2\Delta t\cdot\mathbf S_{i_1,j_1}^{i,j}\in\mathscr G$ for all $i_1,j_1$, then we can obtain $\bar{\mathbf U}^{i,j,n+1}\in\mathscr G$ by the convexity of $\mathscr G$. For the $\mathbf H$ terms,  we have that they all belong to $\mathscr G$ if
$$ \frac{8\lambda_x}{\omega_0}\alpha_{x,0}\le 1,\quad \frac{8\lambda_y}{\omega_0}\alpha_{y,0}\le 1. $$
For the source term, we omit the subscripts $i,j,i_1,j_1,n$ for brevity and define
$$\mathbf U^\star:=
\mathbf{U}+2\Delta t\cdot \mathbf{S}=\left[ \begin{array}{c}
	\rho\\
	m+4\lambda _x\rho \Theta\\
	n+4\lambda _y\rho \Xi\\
	\mathcal{E} +4\lambda _xm\Theta +4\lambda _yn\Xi\\
\end{array} \right]. 
$$
Similar to the 1D case, the density of  state $\mathbf U^\star$ is positive,  and hence $\mathbf U^\star\in\mathscr G$ if and only if its pressure is also positive, i.e.,
$$
\mathcal{E} +4\lambda _xm\Theta +4\lambda _yn\Xi -\frac{\left( m+4\lambda _x\rho \Theta \right) ^2+\left( n+4\lambda _y\rho \Xi \right) ^2}{2\rho}>0.
$$
The solution of above inequality is
\begin{equation}\label{eq:inepp2D}
\lambda _{x}^{2}\Theta ^2+\lambda _{y}^{2}\Xi ^2<\frac{1}{16\left( \gamma -1 \right) \beta}.
\end{equation}
It is easy to verify that \eqref{eq:inepp2D} holds under the CFL condition \eqref{eq:CFLPP2D}. 
\end{proof}

\begin{remark}
    We want to remark that the 2D PP limiter does not increase the total entropy as well.
\end{remark}

\section{Numerical experiments}
\label{sec5}
\setcounter{equation}{0}
\setcounter{table}{0}
\setcounter{figure}{0}

In this section, we present numerical results for the proposed method to demonstrate its high-order accuracy for smooth solutions, as well as its WB, ES, and PP properties. Unless otherwise noted, we only show results with $k = 2$. The ten-stage, fourth-order SSP-RK method developed by Ketcheson \cite{ketcheson2008highly} is used for time discretization, with time step
$$ \Delta t =  \frac{\mathrm{CFL}}{\alpha_{x,0}} \Delta x, \quad \text{or} \quad 
\Delta t =  \frac{\mathrm{CFL}}{\alpha_{x,0}/\Delta x+\alpha_{y,0}/\Delta y} ,\quad\mathrm{CFL}=0.5,$$
and hence the CFL condition for the PP property \eqref{eq:CFLPP} or \eqref{eq:CFLPP2D} is satisfied.
Note that, to demonstrate the reliability and robustness of the proposed scheme, no  slope limiters are employed in any of the tests. Our method, denoted as “WBESPP,” is compared against the following three methods to highlight its advantages in preserving simultaneously the WB, ES, and PP properties. Specifically, for the 1D case, these schemes are defined as follows:
\begin{itemize}
\item[(a)] ``non-WB" method with the semi-discrete formulation  
\begin{equation*}
\frac{\Delta x}{2}\frac{\mathrm d\mathbf U_{i_1}^i}{\mathrm dt}+\sum\limits_{l=0}^k2D_{i_1,l}\mathbf F^S(\mathbf U_{i_1}^i,\mathbf U_{l}^i)+\frac{\tau _{i_1}}{\omega_{i_1}}(\mathbf F_{i_1}^{*,i}-\mathbf F_{i_1}^i)=\frac{\Delta x}{2}\mathbf S(\mathbf U_{i_1}^i,x_i(X_{i_1})).
\end{equation*}
Note that the scheme is ES due to $(\mathbf V_{i_1}^i)^T\mathbf S(\mathbf U_{i_1}^i,x_i(X_{i_1}))=0$. 
The limiter $\Pi_h$ is applied. Hence, the non-WB scheme is ES and PP but not WB.

\item[(b)] ``non-ES" method with the semi-discrete formulation  
\begin{equation*}
\frac{\Delta x}{2}\frac{\mathrm d\mathbf U_{i_1}^i}{\mathrm dt}+\sum\limits_{l=0}^kD_{i_1,l}\mathbf F_{l}^i+\frac{\tau _{i_1}}{\omega_{i_1}}(\mathbf F_{i_1}^{*,i}-\mathbf F_{i_1}^i)=\tilde{\mathbf S}_{i_1}^{i},
\end{equation*}
where 
\begin{equation*}
    \tilde{\mathbf S}_{i_1}^i=\left( 0,\ \rho_{i_1}^i\tilde\Theta_{i_1}^i,\ m_{i_1}^i\tilde\Theta_{i_1}^i \right),\quad \tilde \Theta_{i_1}^i = \frac{1}{\rho_{i_1}^{e,i}} \sum\limits_{l=0}^k D_{i_1,l} F_{2,l}^{i_1}.
\end{equation*}
The PP limiter $\Pi_h$ is employed. The non-ES scheme is WB and PP but not ES.  

\item[(c)] ``non-PP" method with the same  formulation as  \eqref{eq:nodalES}. 
While the PP limiter $\Pi_h$ is not used. Hence, the non-PP scheme is WB and ES but not PP. 

\end{itemize}

\subsection{One-dimensional tests}

\ 

\textit{Example} 5.1 (Well-balancedness test.) In this example, we consider different steady states and test the well-balanced property of the proposed scheme. We take $\gamma = 5/3$ and $\Omega = [0,2]$ in this example. Under the linear gravitational potential $\phi_x=1$, we consider both the isothermal and the isentropic steady state solutions, which are respectively given by
\begin{equation*}
\begin{aligned}
    \text{Eqbm1:}\quad & \rho ^e(x)=\exp\left(-x\right),\ p^e(x)=\exp \left( -x \right),\\
    \text{Eqbm2:}\quad & 
\rho ^e(x)=\left( \rho _{0}^{\gamma -1}-\frac{1}{K_0}\frac{\gamma -1}{\gamma}gx \right) ^{\frac{1}{\gamma -1}},\ p^e(x)=K_0\left( \rho ^e(x) \right) ^{\gamma}.
\end{aligned}
\end{equation*}
The parameters for Eqbm2 are set as $\rho_0=1,\ g=1,\ K_0=1$. One can easily check that both of them satisfy \eqref{eq:eqbm} and will lead to the exact balance of the flux and the source term. 

First, we take the steady state solution as initial condition and compute the solution until time $T=4$. 
In Tab. \ref{tabwb}, we present the errors of density of different schemes. It is notable that these errors are calculated between the numerical solutions $\mathbf U_h^n$ and the interpolating polynomials of steady state solutions $\mathbf U_h^e$. We can see that errors of our WBESPP scheme are all at the level
of round-off error both for the isothermal equilibrium and the polytropic equilibrium, demonstrating the desired WB property independent of the specific type of the equilibrium.
While the non-WB scheme shows the
expected third order accuracy of the scheme.

Next, we add a small perturbation to Eqbm2 at the left boundary $x=0$ to
show the advantage of WB scheme in simulating the evolution
of such small perturbation. The perturbation is added on the velocity, which is given by
$ u(0,t) = 10^{-6} \sin(4\pi t).$
The solutions are computed until time
$T=1.5$, at which point the waves have not yet propagated to the right boundary. In Fig. \ref{figWB}, we present the pressure perturbation and velocity on $N=200$ meshes against the reference solution computed by the WBPP method \cite{du2024well} on $N=2000$ meshes. 
One can see that the results with the WB scheme agree well with the reference solution, while the results by the non-WB scheme has a large error especially for $x>1.5$. Thus indicates that the WB method is more accurate for resolving small perturbation to steady states.\\

\begin{table}[htb!]
        \centering
        \caption{Example 5.1: One-dimensional well-balancedness test. Errors and orders of density at final time $T = 4$ for $k=2$.}
	\setlength{\tabcolsep}{2.1mm}{
		\begin{tabular}{|c|c|cc|cc|cc|}
			\hline 
   &$N$ & $L^1$ error & order & $L^2$ error & order & $L^\infty$ error & order  \\ \hline
   \multicolumn{8}{|c|}{WBESPP}\\ \hline
   \multirow{4}{*}{Eqbm1}
     &20 &1.67e-15 & -- &2.09e-15 & -- &5.55e-15 &-- \\ 
     &40 &3.10e-15 & -- &3.95e-15 & -- &1.11e-14 &-- \\
     &80 &5.42e-15 & -- &7.04e-15 & -- &2.08e-14 &-- \\
    &160 &1.41e-14 & -- &1.66e-14 & -- &5.96e-14 &-- \\
			\hline
   \multirow{4}{*}{Eqbm2} 
     &20 &1.56e-15 & -- &2.57e-15 & -- &9.21e-15 &-- \\ 
     &40 &4.43e-15 & -- &6.34e-15 & -- &2.13e-14 &-- \\
     &80 &8.08e-15 & -- &1.10e-14 & -- &3.54e-14 &-- \\
    &160 &1.56e-14 & -- &2.20e-14 & -- &6.94e-14 &-- \\
    \hline 
    \multicolumn{8}{|c|}{non-WB}\\ \hline
   \multirow{4}{*}{Eqbm1}
     &20 &7.01e-06 & -- &9.02e-06 & -- &3.30e-05 &-- \\
     &40 &9.03e-07 & 2.96 &1.17e-06 & 2.95 &4.44e-06 &2.89 \\ 
     &80 &1.15e-07 & 2.98 &1.49e-07 & 2.97 &5.77e-07 &2.95 \\
    &160 &1.44e-08 & 2.99 &1.88e-08 & 2.99 &7.35e-08 &2.97 \\
			\hline
   \multirow{4}{*}{Eqbm2}
     &20 &3.34e-06 & --&3.75e-06 & -- &9.44e-06 &-- \\
     &40 &4.36e-07 & 2.94 &4.95e-07 & 2.92 &1.29e-06 &2.87 \\
     &80 &5.58e-08 & 2.97 &6.36e-08 & 2.96 &1.69e-07 &2.93 \\
    &160 &7.05e-09 & 2.98 &8.07e-09 & 2.98 &2.17e-08 &2.96 \\    
    \hline
	\end{tabular}} 
    \label{tabwb}
\end{table}

\begin{figure}[htbp!]
	\centering
    \subfigure[Pressure perturbation.]{
		\includegraphics[width=0.4\linewidth]{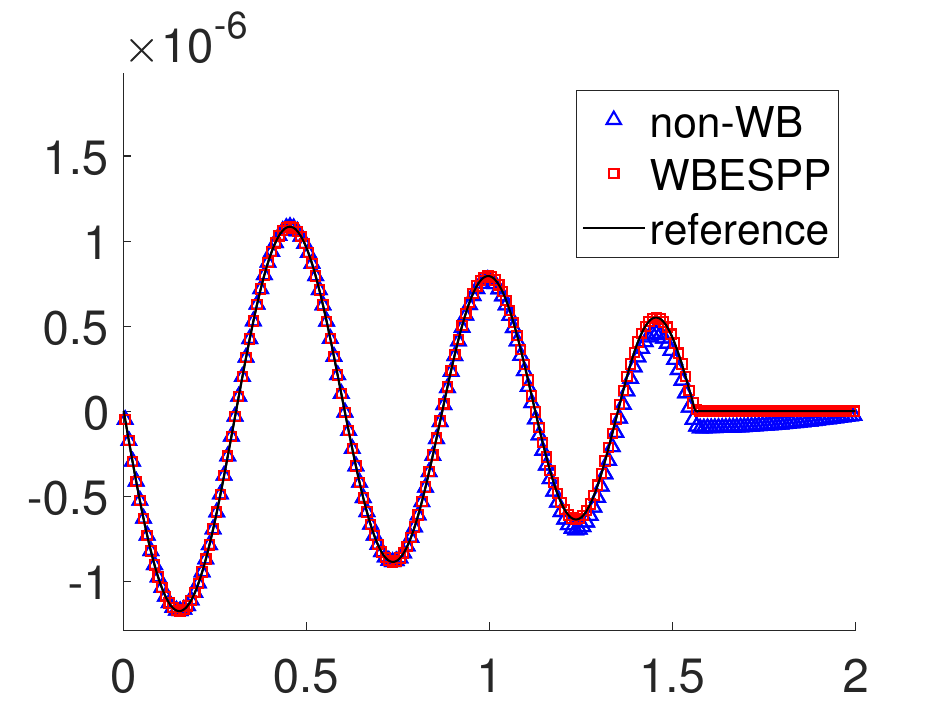}}
	\subfigure[Velocity.]{
		\includegraphics[width=0.4\linewidth]{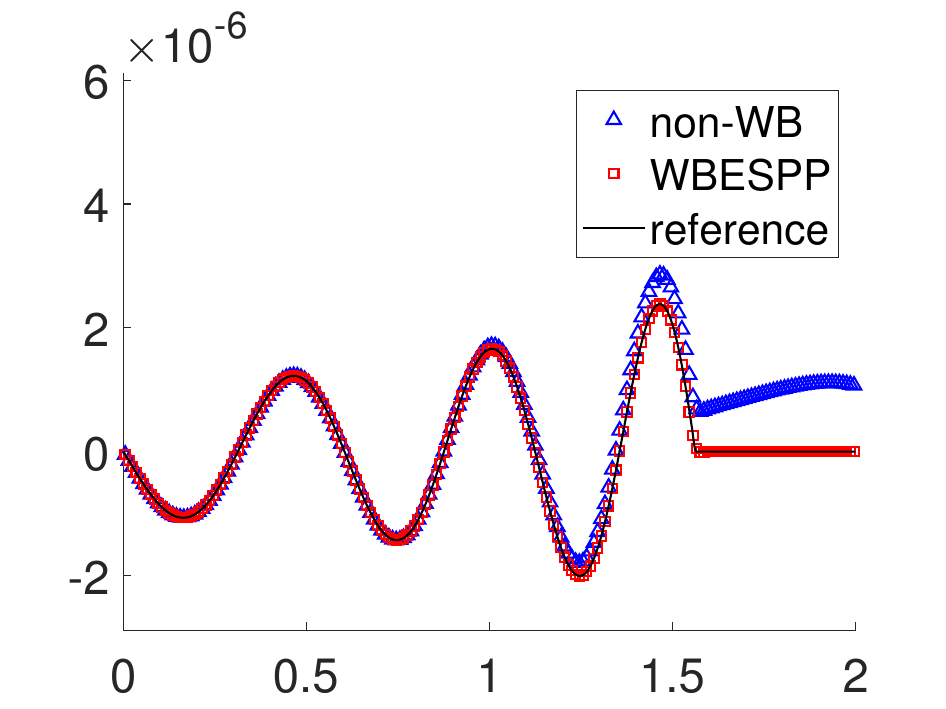}}
	\caption{Example 5.1: One-dimensional well-balancedness test with small velocity perturbation. The numerical solution at $T = 1.5$ on $N=200$ meshes.}
   \label{figWB}
\end{figure}

\textit{Example} 5.2 (Sod shock tube.) 
Here, we test the 1D Sod-like shock tube problem with the gravitational potential $\phi_x = 1$, which is a standard test example for Euler equation, and is originally introduced in \cite{sod1978survey}. The computational domain is taken as $\Omega = [-1,1]$ with reflective boundaries on both sides, and the initial data is given by
$$
\left( \rho ,u,p \right) =\begin{cases}
	\left( 1,0,1 \right) , & x<0,\\
	\left( 0.125,0,0.1 \right) , & x\ge 0.\\
\end{cases}
$$
In Fig. \ref{figSod}, we present the result of density at $T = 0.4$ on $N=200$ meshes. The reference solution is computed by the WBPP method \cite{du2024well} with TVB limiter on $N=1200$ meshes. 
Due to the entropy stability, we can see the discontinuities are captured sharply and stably, although we do not add the shock limiter. 
While the computation will break down before the final time without ES treatment. In Fig. \ref{figSodU}, we present the time evolution of total entropy for our WBESPP scheme and non-ES scheme. It can be observed that the ES treatment indeed dissipates the total entropy.\\

\begin{figure}[htbp!]
	\centering
    \subfigure[Density.]{
		\includegraphics[width=0.31\linewidth]{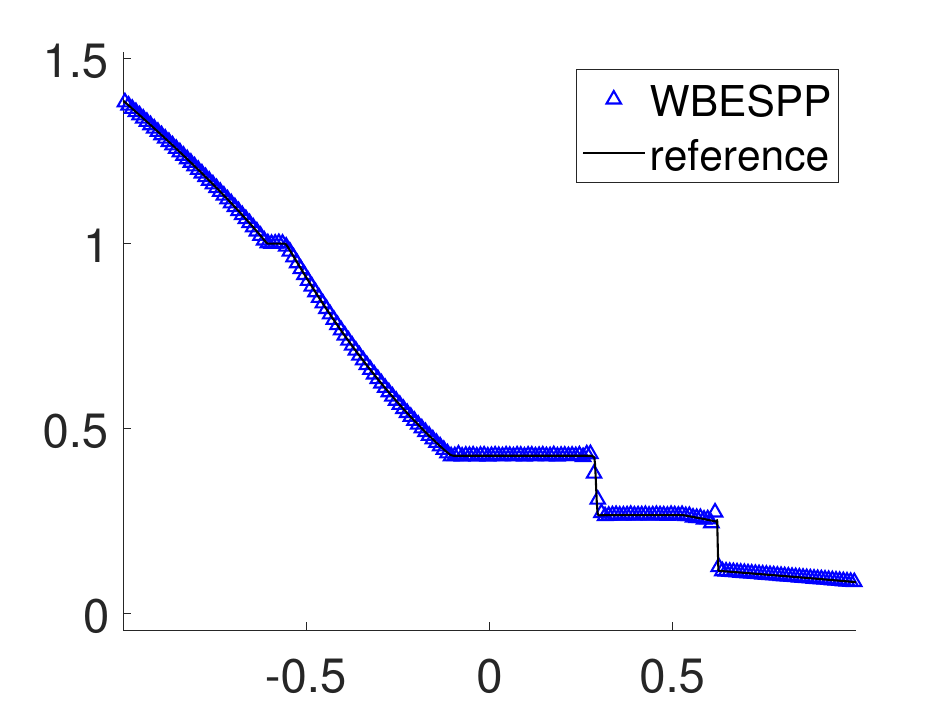}}
	\subfigure[Velocity.]{
		\includegraphics[width=0.31\linewidth]{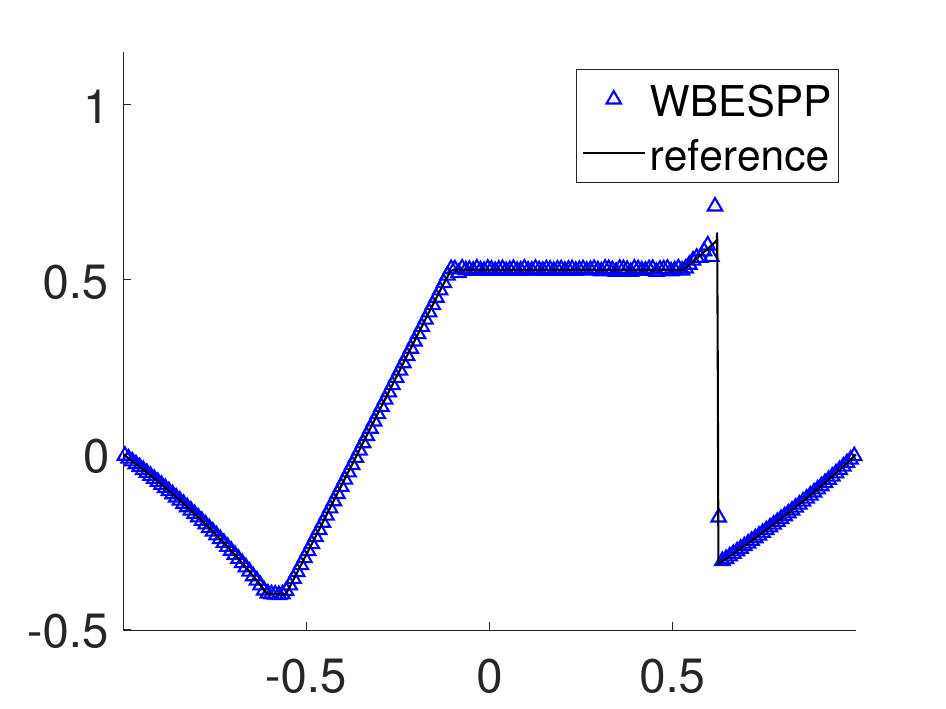}}
    \subfigure[Pressure.]{
		\includegraphics[width=0.31\linewidth]{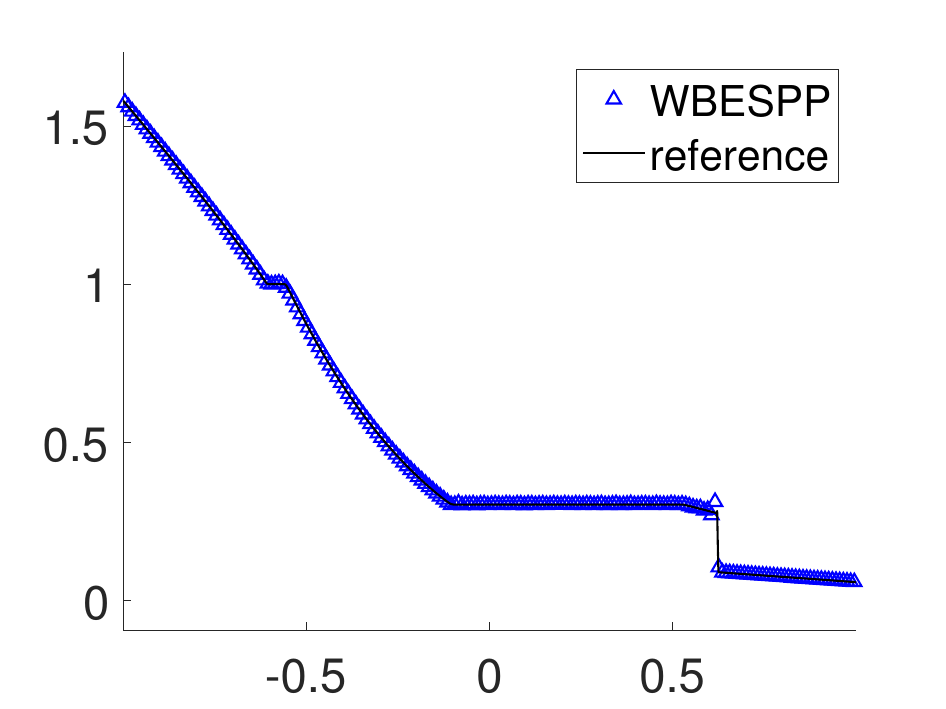}}
	\caption{Example 5.2: One-dimensional Sod-like shock tube. The numerical solution at $T = 0.4$ on $N=200$ meshes.}
   \label{figSod}
\end{figure}

\begin{figure}[htbp!]
	\centering
    \subfigure[non-ES.]{
		\includegraphics[width=0.4\linewidth]{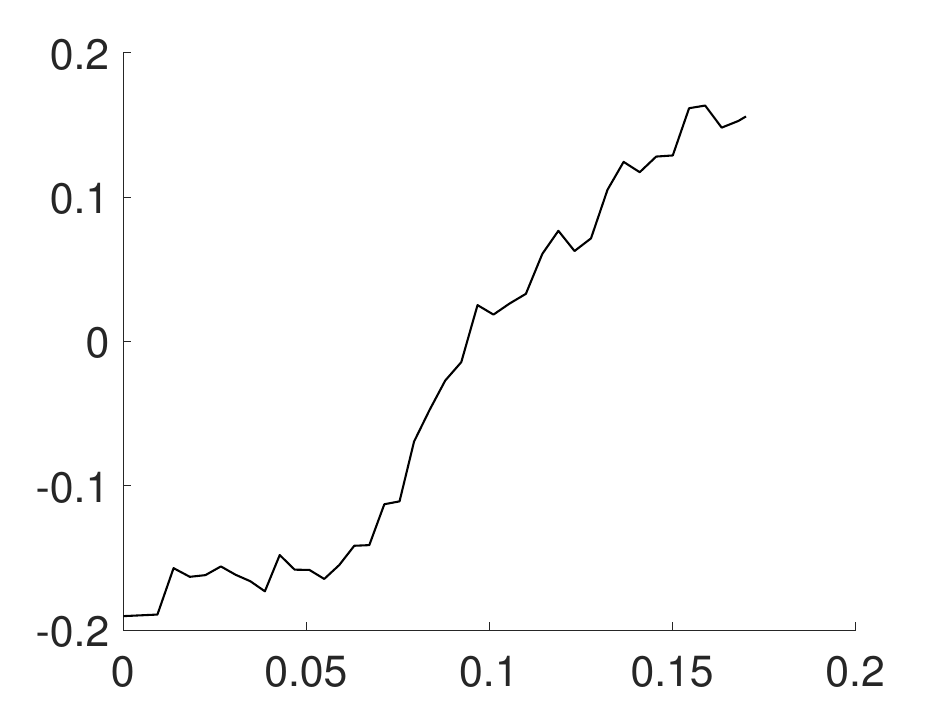}}
	\subfigure[WBESPP.]{
		\includegraphics[width=0.4\linewidth]{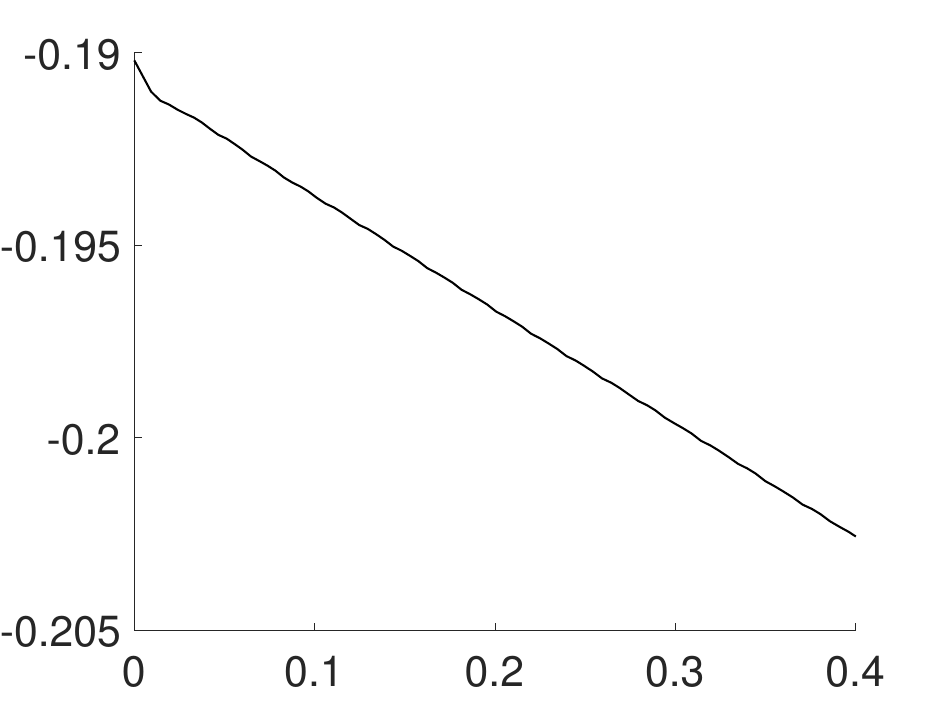}}
	\caption{Example 5.2: One-dimensional Sod-like shock tube. The evolution of total entropy on $N=200$ meshes. }
   \label{figSodU}
\end{figure}

\textit{Example} 5.3 (Double rarefaction.) 
We consider the double rarefaction problem with the gravitational potential $\phi_x=x$. This example includes extreme low density and pressure. The computational domain is taken as $\Omega = [-1,1]$ with outflow boundaries, and the initial data is given as
$$
\left( \rho ,u,p \right) =\begin{cases}
	\left( 7,-1,0.2 \right) , & x<0,\\
	\left( 7,1,0.2 \right) , & x\ge 0.\\
\end{cases}
$$
We run the simulation until $T=0.6$ on $N=800$ meshes, and the results are shown in Fig. \ref{figDRF}. The reference solution is computed by the WBPP method \cite{du2024well} on $N=1600$ meshes. 
We can see that the structure of solution is resolved well, while the density and pressure are preserved to be positive. We want to remark that the computation will break down at the first step without the PP treatment since the negative density and pressure are introduced.\\

\begin{figure}[htbp!]
	\centering
    \subfigure[Density.]{
		\includegraphics[width=0.31\linewidth]{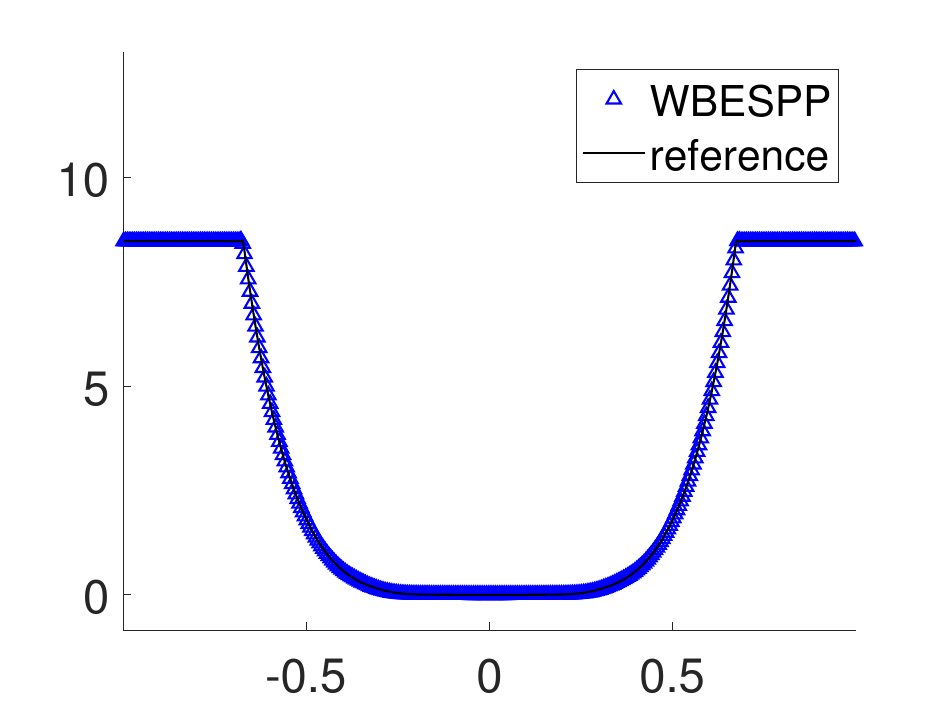}}
	\subfigure[Momentum.]{
		\includegraphics[width=0.31\linewidth]{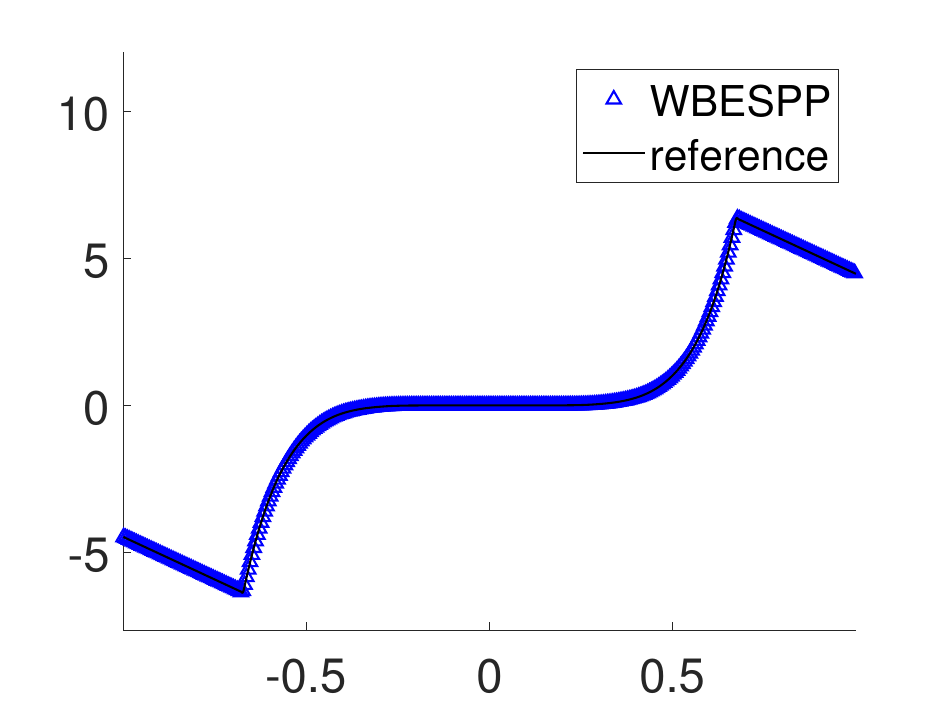}}
	\subfigure[Energy.]{
		\includegraphics[width=0.31\linewidth]{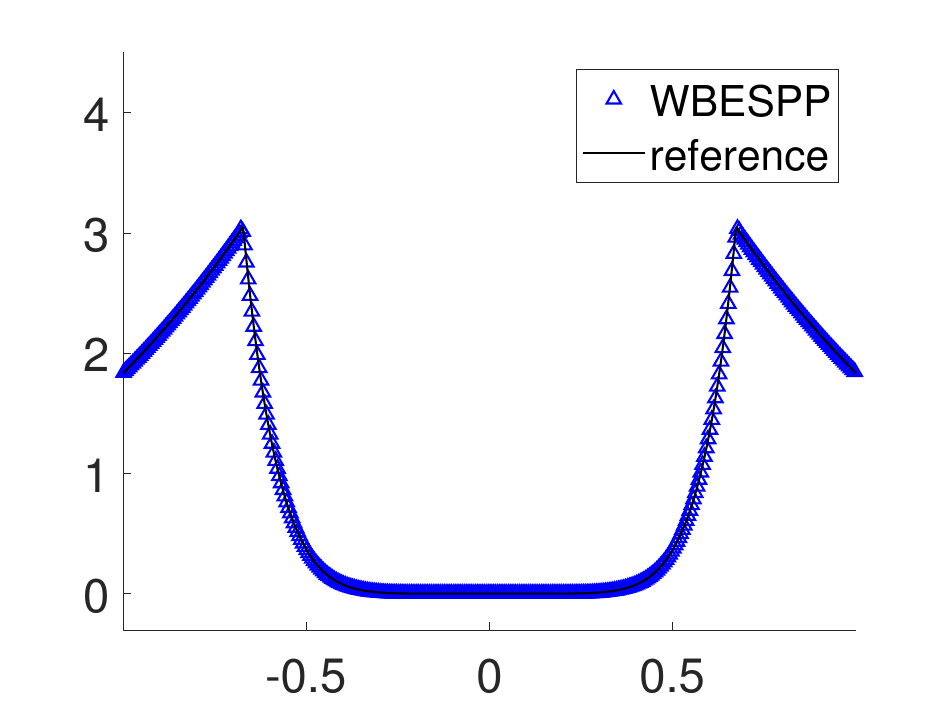}}
	\caption{Example 5.3: One-dimensional double rarefaction problem. The numerical solution at $T = 0.6$ on $N=800$ meshes.}
   \label{figDRF}
\end{figure}

\subsection{Two-dimensional tests}

\ 

\textit{Example} 5.4 (Well-balancedness test.) We consider the two-dimensional isothermal steady state solution in \cite{xing2013high} under the linear gravitational potential $\phi_x=\phi_y=g$. The exact steady state solution is given by
$$
\rho ^e=\rho _0\exp \left( -\frac{\rho _0g}{p_0}\left( x+y \right) \right) ,\quad \mathbf{u}^e=\mathbf{0},\quad p^e=p_0\exp \left( -\frac{\rho _0g}{p_0}\left( x+y \right) \right) ,
$$
where the parameters are $\rho_0=1.21,\ p_0=1,\ g=1$. The computational domain is taken as $\Omega=[0,1]\times[0,1]$.  

We first take the initial data as the steady state solution and compute the solution until $T=1$. In Tab. \ref{tabwb2D1}, we present the errors of density, demonstrating that our proposed WBESPP scheme also maintains the balance errors at machine level for the 2D problem.

Next, we add a small perturbation on pressure
$$ p=p^e+0.001 \exp\left( -100((x-0.3)^2+(y-0.3)^2) \right). $$
In Fig. \ref{figWB2D}, we present the results of density perturbation and pressure perturbation at $T=0.15$ on mesh with $N_x\times N_y=100\times 100$. It is observed that the non-WB scheme can not capture those small perturbations well, while the WB method resolves them accurately.\\

\begin{table}[htb!]
        \centering
        \caption{Example 5.4: Two-dimensional well-balancedness test. Errors and orders of density at final time $T = 1$.}
	\setlength{\tabcolsep}{2.1mm}{
		\begin{tabular}{|c|c|cc|cc|cc|}
			\hline 
   &$N_x=N_y$ & $L^1$ error & order & $L^2$ error & order & $L^\infty$ error & order  \\ \hline
   \multirow{4}{*}{WBESPP}
     &20 &7.08e-15 & -- &8.05e-15 & -- &2.80e-14 &-- \\
     &40 &1.40e-14 & -- &1.56e-14 & -- &5.88e-14 &-- \\
     &80 &2.81e-14 & -- &3.13e-14 & -- &1.12e-13 &-- \\
    &160 &5.72e-14 & -- &6.38e-14 &-- &2.19e-13 &-- \\

			\hline
   \multirow{4}{*}{non-WB}
     &20 &1.31e-06 & -- &1.53e-06 & -- &9.10e-06 &-- \\
     &40 &1.65e-07 & 2.98 &1.93e-07 & 2.98 &1.21e-06 &2.91 \\
     &80 &2.08e-08 & 2.99 &2.43e-08 & 2.99 &1.58e-07 &2.94 \\
    &160 &2.60e-09 & 3.00 &3.05e-09 & 3.00 &2.01e-08 &2.97 \\

			\hline
	\end{tabular}} 
    \label{tabwb2D1}
\end{table}

\begin{figure}[htb!]
	\centering
    \subfigure[Density perturbation, non-WB.]{
		\includegraphics[width=0.4\linewidth]{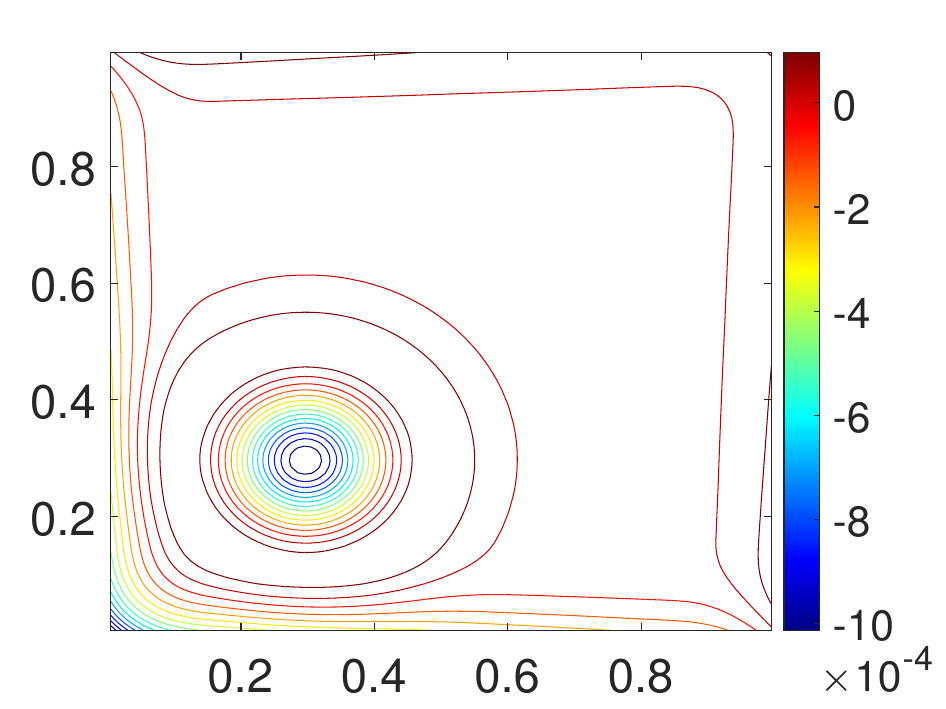}}
	\subfigure[Pressure perturbation, non-WB.]{
		\includegraphics[width=0.4\linewidth]{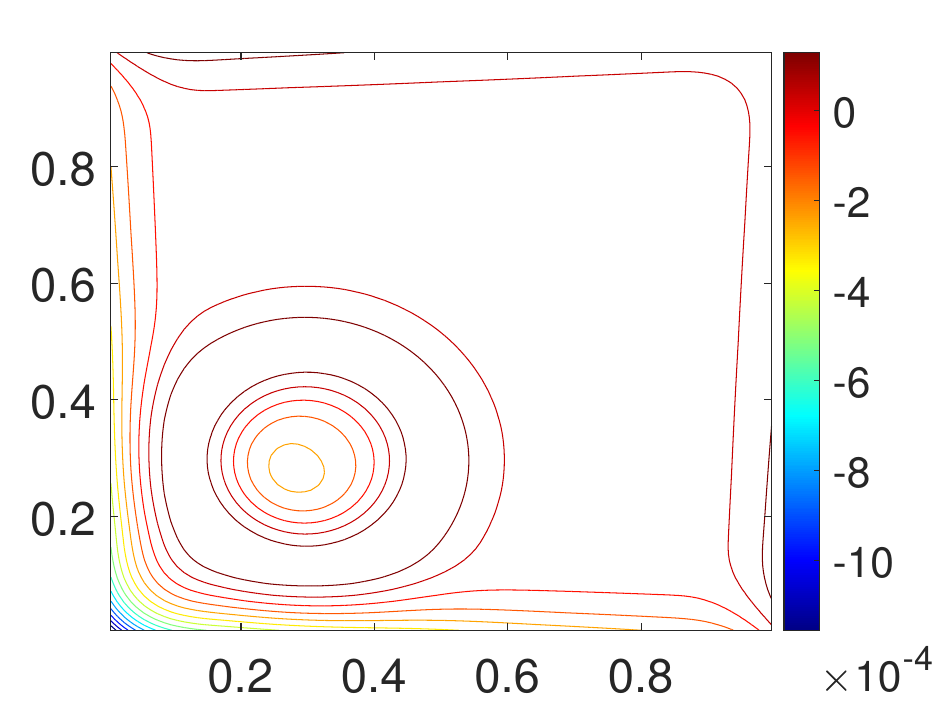}}
    \subfigure[Density perturbation, WBESPP.]{
		\includegraphics[width=0.4\linewidth]{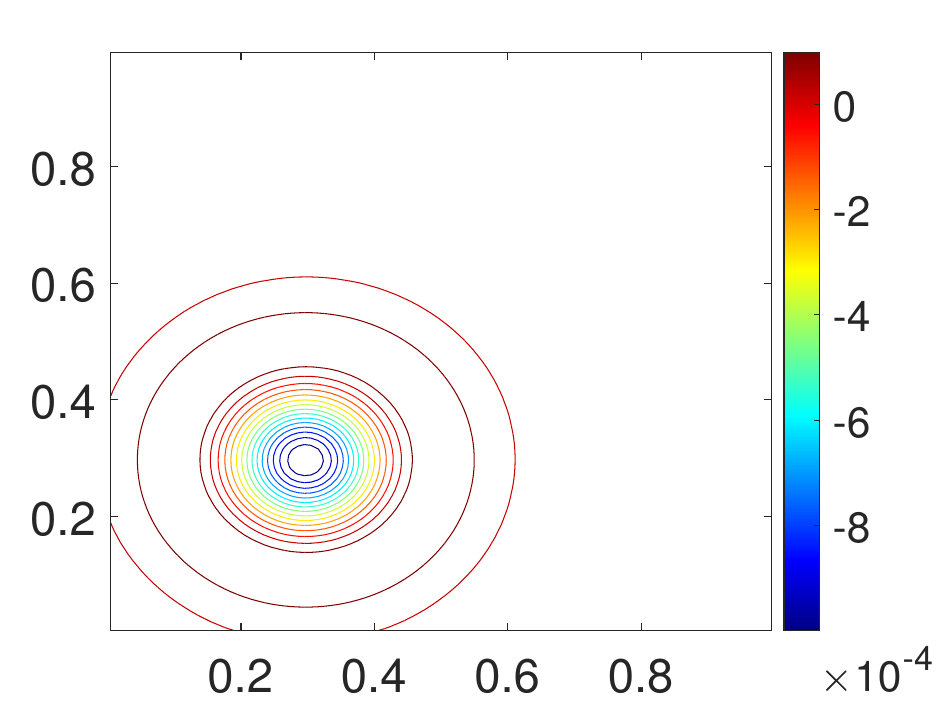}}
	\subfigure[Pressure perturbation, WBESPP.]{
		\includegraphics[width=0.4\linewidth]{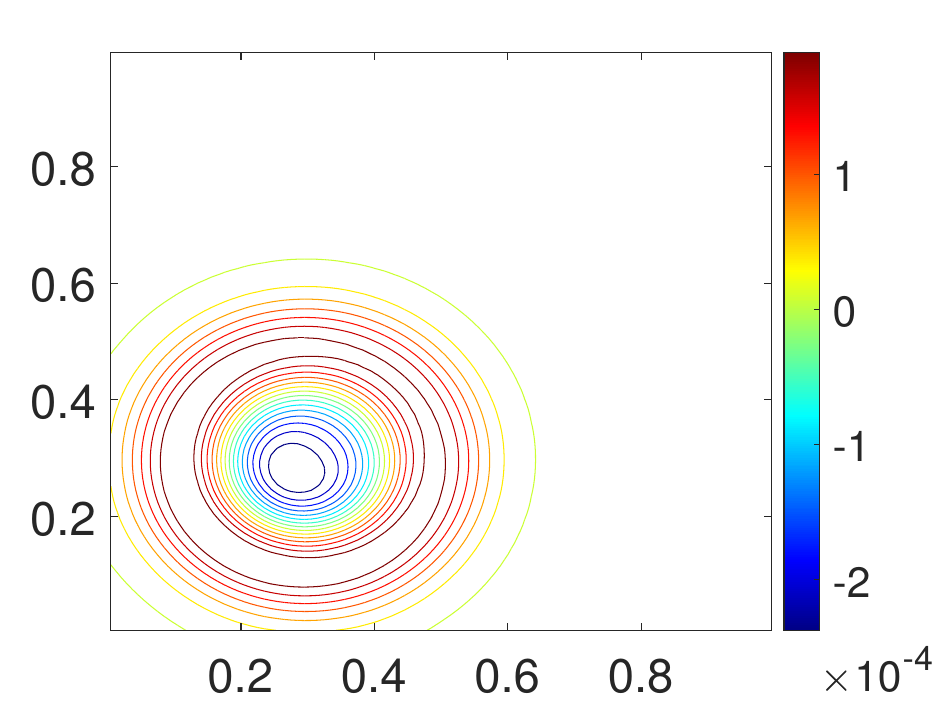}}
	\caption{Example 5.4: Two-dimensional well-balancedness test. The numerical solution at $T = 0.15$ on $N_x\times N_y=100\times 100$ meshes. 15 contour lines are used.}
   \label{figWB2D}
\end{figure}

\textit{Example} 5.5 (Accuracy test.) 
In this example, we test the accuracy of the scheme with the gravitational potential $\phi_x=\phi_y=1$. The computational domain is taken as $\Omega=[0,2\pi]\times[0,2\pi]$, and the exact smooth solution is given by
$$
\rho =1+0.2\sin \left( x+y-2t \right),\quad u=v=1,\quad p=20-x-y+2t+0.2\cos(x+y-2t).
$$
We run the simulation up to $T = 0.5$. In Tab. \ref{tabwb2D2}, we present the errors and orders of density for $k = 1,2,3$, and the optimal $(k+1)$-th convergence rates are observed.\\

\begin{table}[htb!]
        \centering
        \caption{Example 5.5: Two-dimensional accuracy test. Errors and orders of density at final time $T = 0.5$.}
	\setlength{\tabcolsep}{2.1mm}{
		\begin{tabular}{|c|c|cc|cc|cc|}
			\hline 
   &$N_x=N_y$ & $L^1$ error & order & $L^2$ error & order & $L^\infty$ error & order  \\ \hline 
   \multirow{4}{*}{$k=1$}
     &20 &2.73e-03 & -- &3.16e-03 & -- &6.69e-03 &-- \\
     &40 &7.00e-04 & 1.96 &8.06e-04 & 1.97 &1.94e-03 &1.79 \\
     &80 &1.76e-04 & 1.99 &2.02e-04 & 2.00 &5.20e-04 &1.90 \\
    &160 &4.40e-05 & 2.00 &5.03e-05 & 2.00 &1.34e-04 &1.96 \\
			\hline
   \multirow{4}{*}{$k=2$}
     &20 &2.77e-04 & -- &3.53e-04 & -- &1.25e-03 & -- \\
     &40 &5.21e-05 & 2.41 &6.65e-05 & 2.41 &2.33e-04 &2.43 \\
     &80 &8.30e-06 & 2.65 &1.06e-05 & 2.65 &3.63e-05 &2.68 \\
    &160 &1.16e-06 & 2.84 &1.48e-06 & 2.84 &4.95e-06 &2.87 \\
    \hline
  \multirow{4}{*}{$k=3$}
     &20 &1.51e-06 & -- &2.07e-06 & -- &1.74e-05 &-- \\
     &40 &7.61e-08 & 4.31 &1.05e-07 & 4.30 &9.16e-07 &4.25 \\
     &80 &5.09e-09 & 3.90 &7.10e-09 & 3.89 &6.25e-08 &3.87 \\
    &160 &2.47e-10 & 4.36 &4.36e-10 & 4.02 &3.12e-09 &4.32 \\
   \hline
	\end{tabular}} 
    \label{tabwb2D2}
\end{table}

\textit{Example} 5.6 (Double rarefaction.) Here, we consider the two-dimensional double rarefaction test in \cite{jiang2022positivity}, which is used to demonstrate the positivity-preserving property of a scheme. The computational domain is $\Omega=[-0.5,0.5]\times[-0.5,0.5]$ with the gravitational potential $\phi=0.5(x^2+y^2)$. The initial data is given by
$$
\rho =\exp \left( -\phi \left( x,y \right) /0.4 \right) ,\quad p=0.4\exp \left( -\phi \left( x,y \right) /0.4 \right), 
$$
$$
u=\left\{\begin{array}{rr}
	-2,&x\le 0,
	 \\2,&x>0,
\end{array}\right. \quad v=0.
$$
We run the simulation until $T=0.1$ on $N_x\times N_y=200\times 200$ meshes, and the results are shown in Fig. \ref{figDRF2D}. It can be seen that the WBESPP scheme can compute this problem stably. We note that for the non-PP scheme, the computation will break down at first several steps.\\

\begin{figure}[htbp!]
	\centering
    \subfigure[Density.]{
		\includegraphics[width=0.31\linewidth]{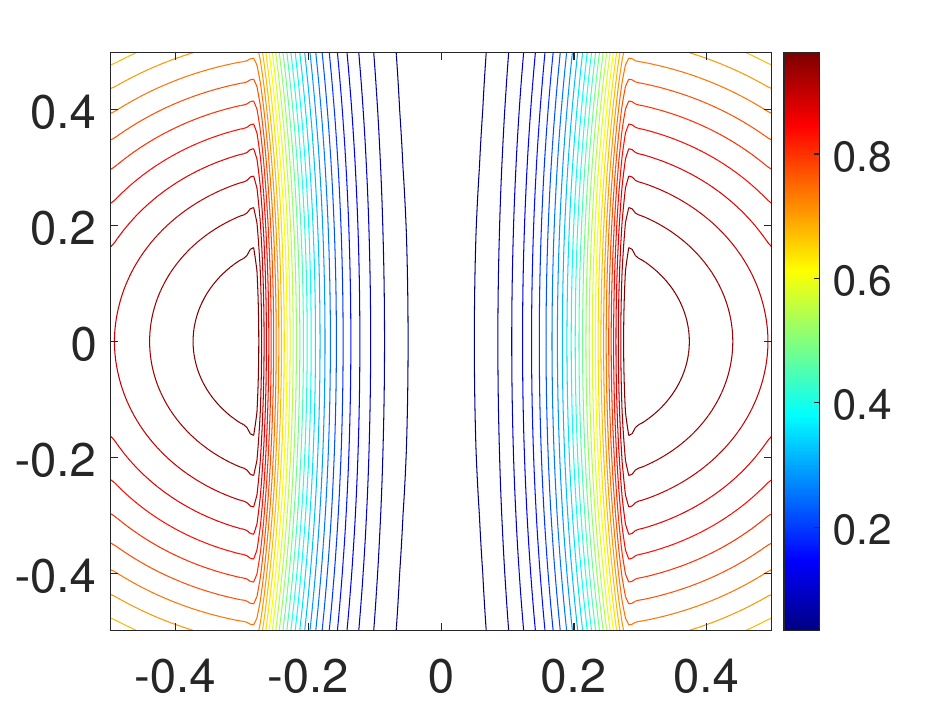}}
	\subfigure[$m=\rho u$.]{
		\includegraphics[width=0.31\linewidth]{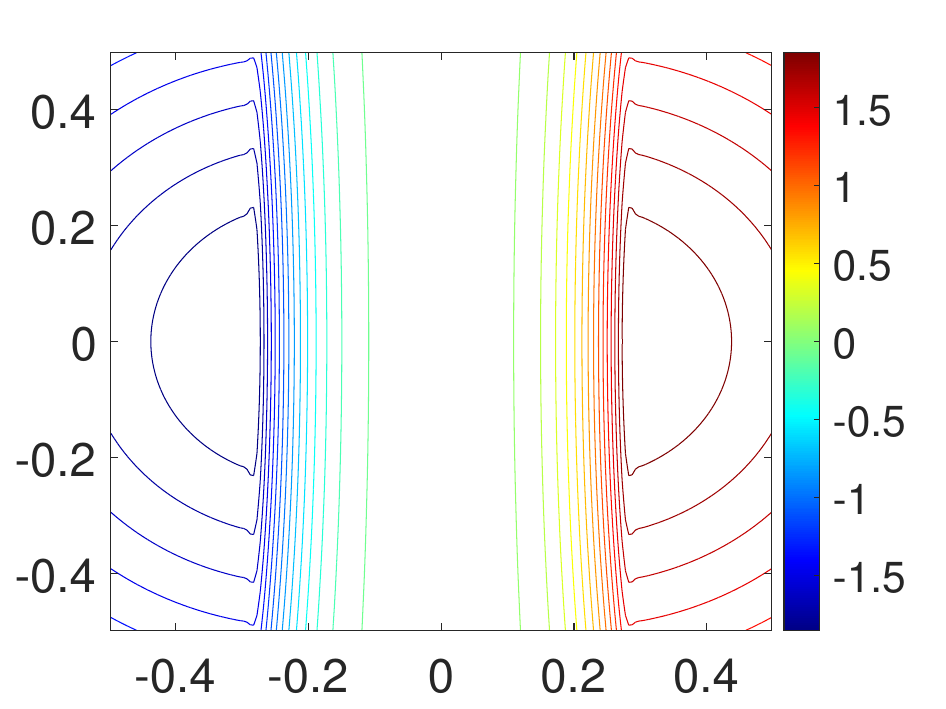}}
    \subfigure[Pressure.]{
		\includegraphics[width=0.31\linewidth]{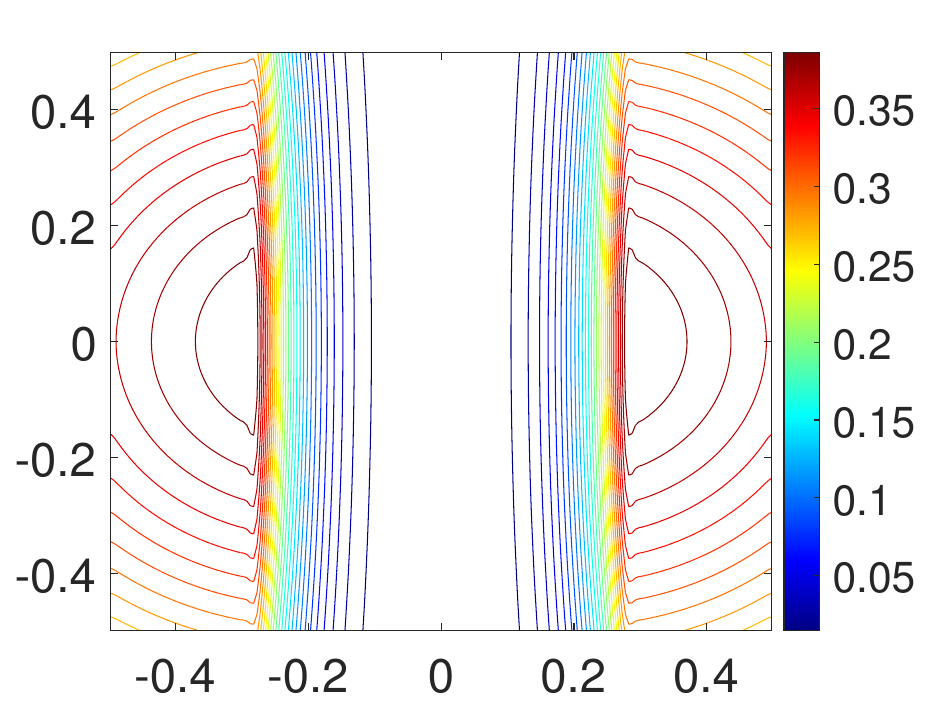}}
	\caption{Example 5.6: Two-dimensional double rarefaction test. The numerical solution at $T = 0.1$ on $N_x\times N_y=200\times 200$ meshes. 30 contour lines are used.}
   \label{figDRF2D}
\end{figure}

\textit{Example} 5.7 (Radial Rayleigh-Taylor instability.) In this example, we simulate the radial Rayleigh-Taylor instability problem with the  gravitational potential $\phi=r=\sqrt{x^2+y^2}$. The initial data is
$$
\rho =\left\{\begin{array}{ll}
	\exp \left( -r \right) ,&r<r_0,\\
	\exp \left( -\frac{r}{\alpha}+r_0\frac{1-\alpha}{\alpha} \right) ,&r>r_0,\\
\end{array}\right.\quad p=\left\{\begin{array}{ll}
	\exp \left( -r \right) , & r<r_i,\\
	\frac{1}{\alpha}\exp \left( -\frac{r}{\alpha}+r_0\frac{1-\alpha}{\alpha} \right) , & r>r_i,\\
\end{array}\right.
$$
where the parameters are given by $$r_i=\eta(1+\cos(k\theta)),\quad \alpha=\exp(r_0)/(\exp(r_0)+\Delta_\rho),\quad \theta=\arctan(y/x),$$
and $r_0=6,\ \eta=0.02,\ \Delta_\rho=0.1,\ k=20$. The equilibrium state is set to $\rho^e=p^e=\exp(-r)$. We can see that in the regions $r < r_0(1-\eta)$ and $r>r_0(1+\eta)$, the initial condition is
in stable equilibrium. But due to the discontinuous density, a Rayleigh-Taylor instability will develop
near the interface $r=r_i(\theta)$ at $t\approx 2.5$. Meanwhile, the solution is closed to the equilibrium state at the location away from the interface.

In Fig. \ref{figRT}, we plot the result of density perturbation at $T=2.9,\ 3.8,\ 5$ on $N_x\times N_y=240\times 240$ meshes. One can see that both two schemes produce relatively accurate results with the instabilities concentrated near the location of the interface. Thanks to the ES property, the computation is stable without applying the shock limiter. Moreover, it is notable that the computation will break down at $t=2.41$ without the ES treatment.

However, it seems that there are some non-physical structures appearing in the central area by non-WB scheme. The enlarged results at $T=5$ near the center are shown in Fig. \ref{figRT2}. It can be seen that the perturbation of the non-WB scheme is about $10^{-3}$ level 
without the ring structure,  while the level is $10^{-4}$ for WBESPP scheme, 
indicates that our scheme can better retain the equilibrium solution away from the interface.\\

\begin{figure}[htbp!]
	\centering
    \subfigure[$T=2.9$, non-WB.]{
		\includegraphics[width=0.31\linewidth]{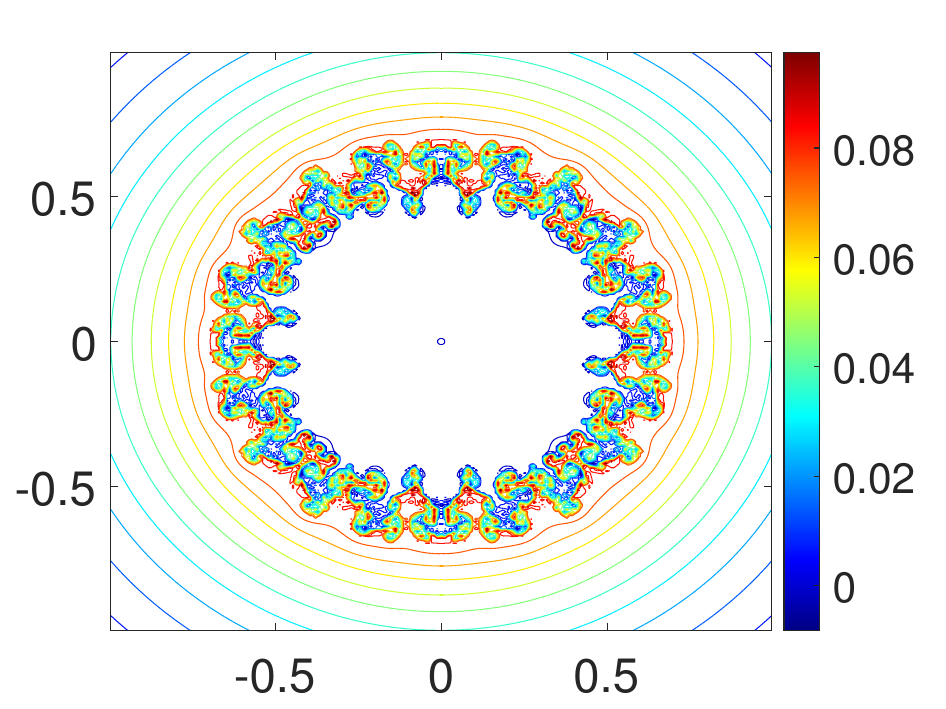}}
    \subfigure[$T=3.8$, non-WB.]{
		\includegraphics[width=0.31\linewidth]{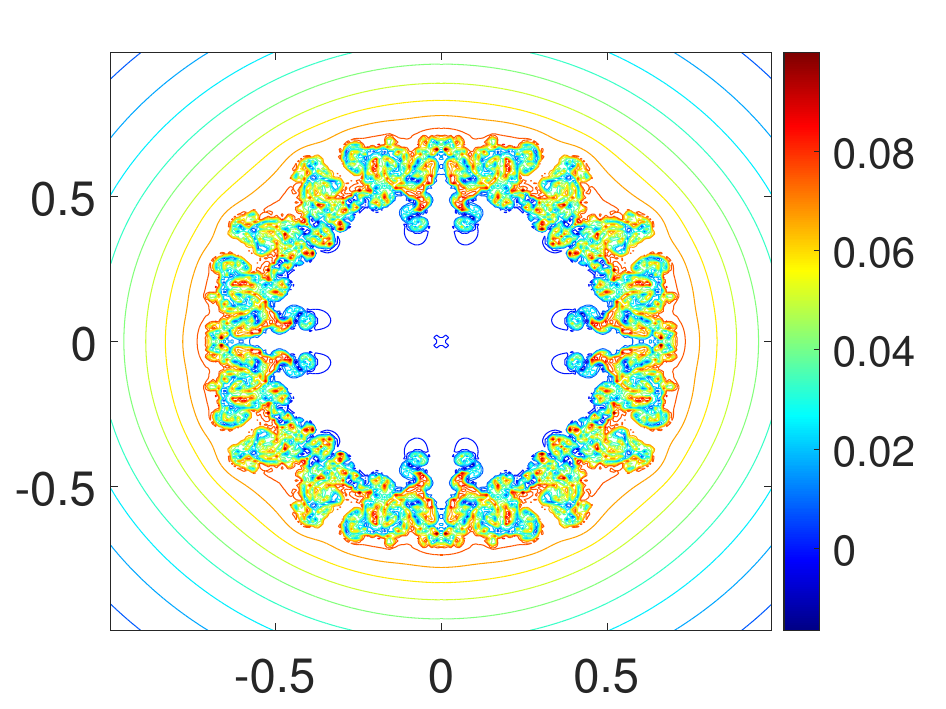}}
    \subfigure[$T=5$, non-WB.]{
		\includegraphics[width=0.31\linewidth]{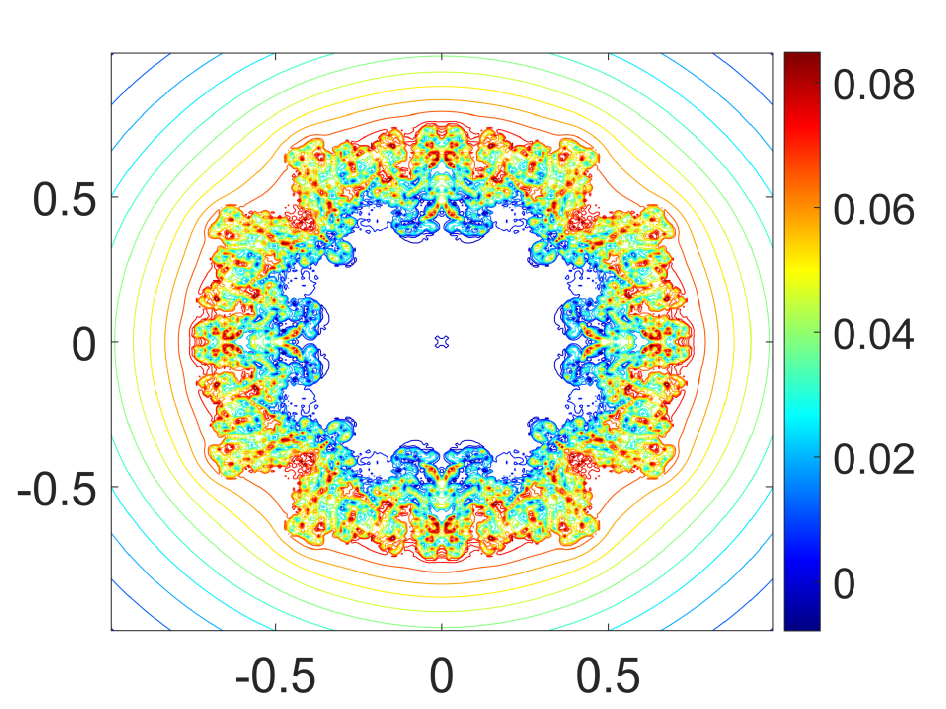}}
    \subfigure[$T=2.9$, WBESPP.]{
		\includegraphics[width=0.31\linewidth]{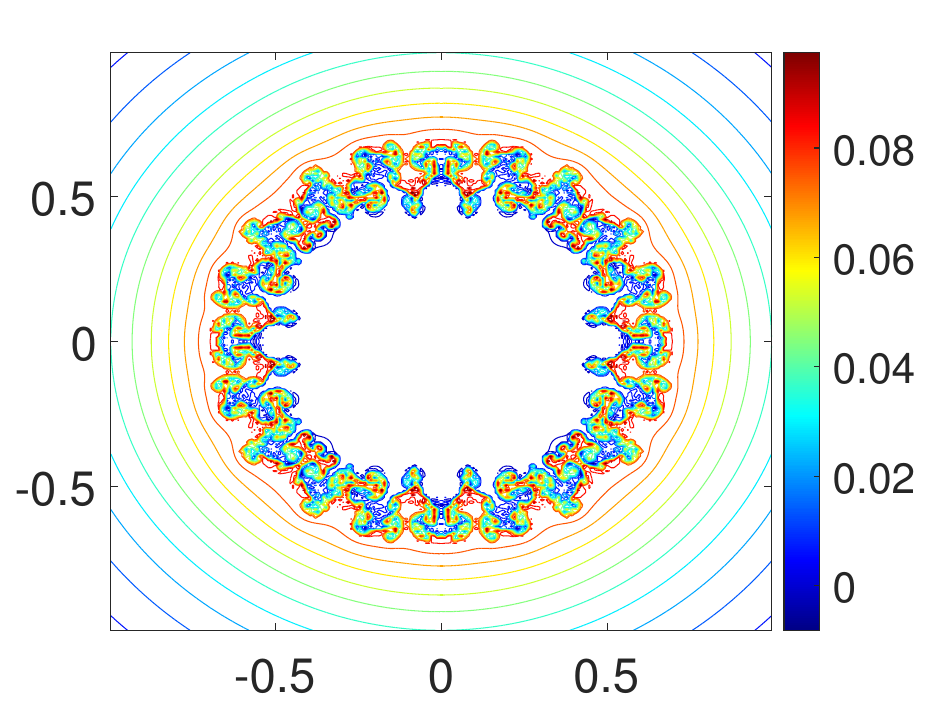}}
    \subfigure[$T=3.8$, WBESPP.]{
		\includegraphics[width=0.31\linewidth]{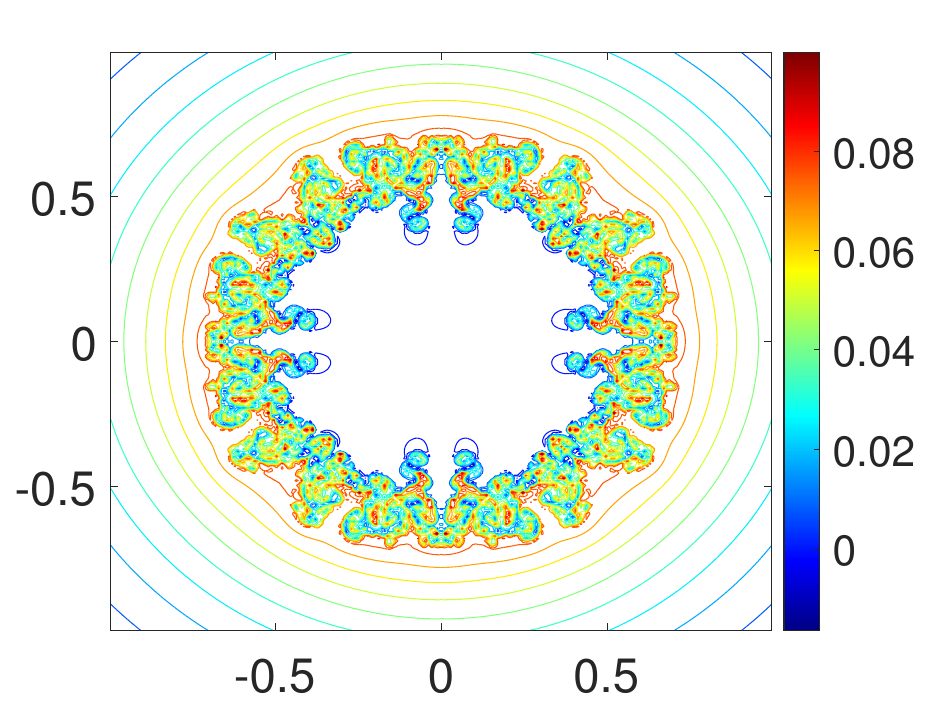}}
	\subfigure[$T=5$, WBESPP.]{
		\includegraphics[width=0.31\linewidth]{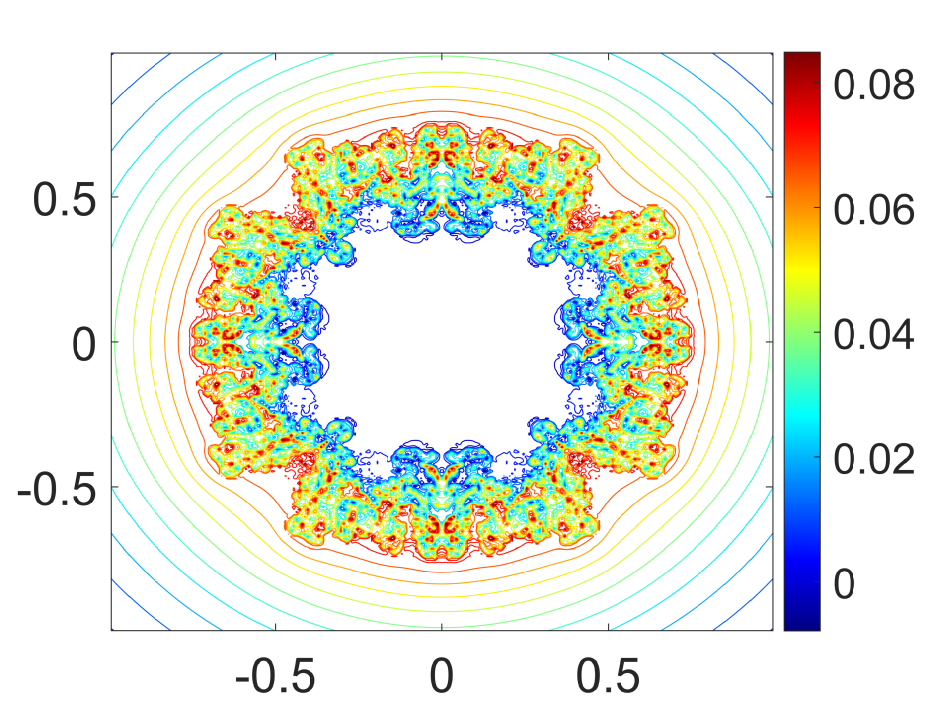}}
	\caption{Example 5.7: Two-dimensional radial Rayleigh-Taylor instability. The density perturbation on $N_x\times N_y=240\times 240$ meshes. 10 contour lines are used.}
   \label{figRT}
\end{figure}

\begin{figure}[htbp!]
	\centering
    	\subfigure[non-WB from -5.7E-3 to 5.0E-3.]{
		\includegraphics[width=0.4\linewidth]{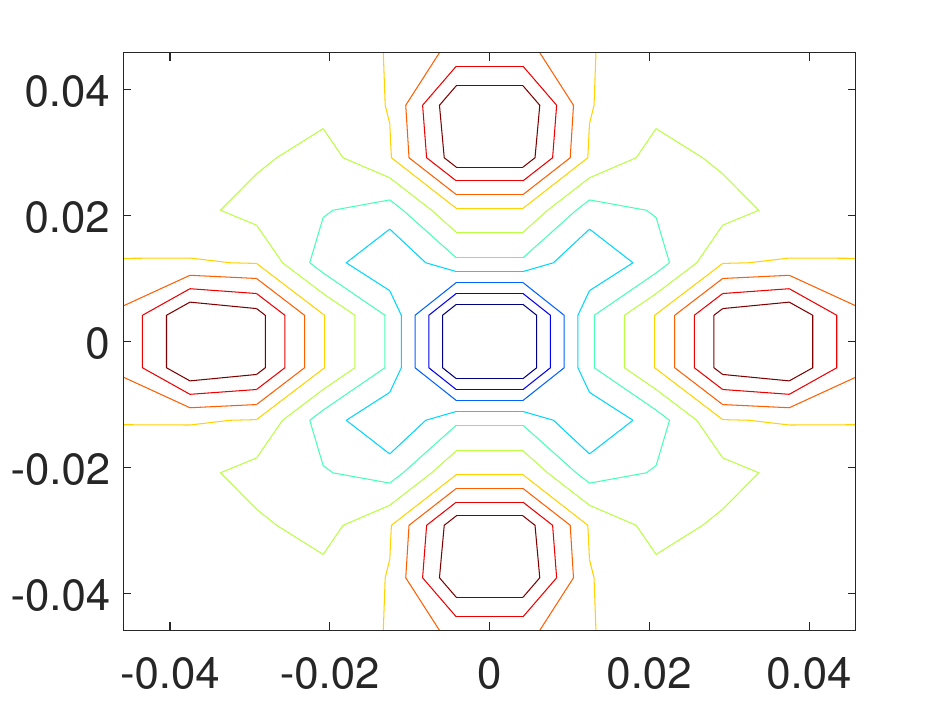}}
	\subfigure[WBESPP from 7.4E-4 to 7.8E-4.]{
		\includegraphics[width=0.4\linewidth]{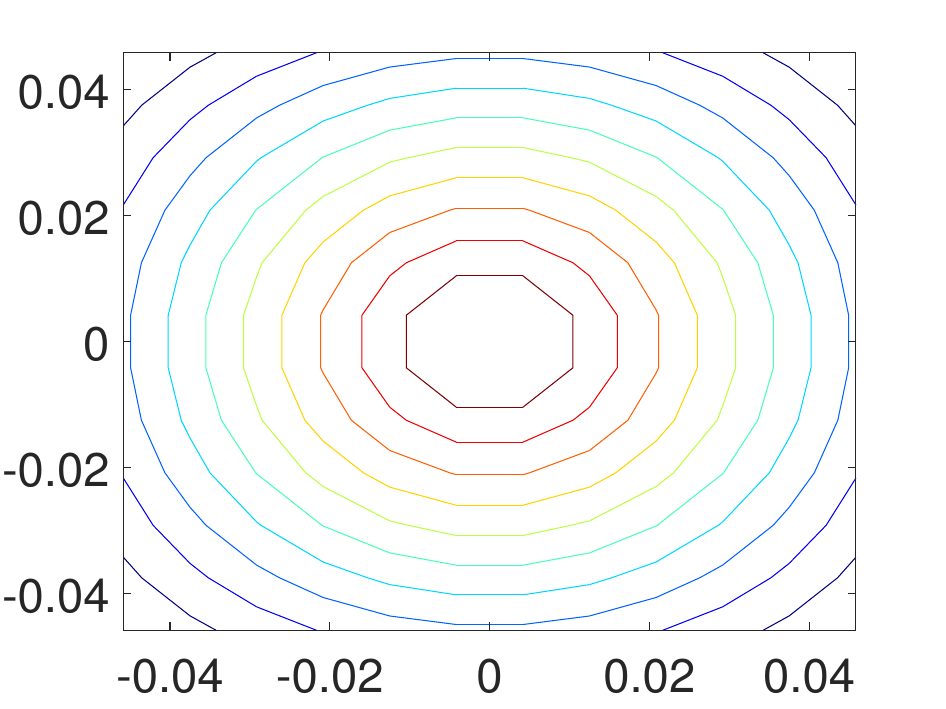}}
	\caption{Example 5.7: Two-dimensional radial Rayleigh-Taylor instability. The density perturbation at $T=5$ around central area on $N_x\times N_y=240\times 240$ meshes. 10 contour lines are used.}
   \label{figRT2}
\end{figure}

\textit{Example} 5.8 (Inertia-gravity wave problem.) Finally, we consider the inertia-gravity wave problem studied in \cite{giraldo2008study}, which is a benchmark test problem arising from atmospheric flows. This problem involves the evolution of a potential temperature perturbation in a channel which is of interest in the validation of numerical weather prediction schemes. The computational domain is a $\Omega = [0,300000]\times[0,10000]\,\mathrm m^2$ rectangle with periodic boundaries on the left and right edges, and reflective boundaries on the top and bottom edges. The gravity function of this problem is $\phi=g\,y, \, g=9.8\,\mathrm m/\mathrm s^2.$ The equilibrium solution is given by
$$ \rho^e=\frac{p_0}{R\Theta^e}\Pi^{\frac{1}{\gamma - 1}},\quad \mathbf u^e=\mathbf 0,\quad p^e=p_0\Pi ^{\frac{\gamma}{\gamma - 1}}, $$
where
$$\Theta^e=T_0\exp\left( \frac{\mathcal N^2}{g}y \right),\quad \Pi=1+\frac{(\gamma - 1)g^2}{\gamma RT_0\mathcal N}\left[\exp\left(-\frac{\mathcal N^2}{g}y\right)-1\right]. $$
The parameters are $ p_0=10^5\,\mathrm N/\mathrm m^2,\ T_0=300\,\mathrm K,\ \mathcal N=0.01/\mathrm s,\ R=287.058\,\,\mathrm J/\mathrm{kg}\,\mathrm K$. Then, a small perturbation is added to the potential  temperature $\Theta$:
$$ \Delta\Theta_0 = \theta_c\sin\left( \frac{\pi y}{h_c} \right)\left[ 1+\left(\frac{x-x_c}{a_c}\right)^2 \right]^{-1}, $$
where $\theta_c=0.01\,\mathrm K,\ h_c=10000\,\mathrm m,\ x_c=100000\,\mathrm m,\ a_c=5000\,\mathrm m$. The initial data for $\rho$ and $p$ are given by
$$ \rho = \frac{p_0}{R(\Theta^e+\Delta\Theta_0)}\Pi^{\frac{1}{\gamma - 1}},\quad p=p^e, $$
and the velocity is set to $u=20\,\mathrm m/\mathrm s,\ v=0$. In Fig. \ref{figIG}, we present the result of the potential temperature perturbation $\Delta\Theta$ at $T=3000\,\mathrm s$ on $N_x\times N_y=1200\times 40$ meshes. Compare to the results in \cite{giraldo2008study}, one can see that the evolution of potential temperature perturbation is resolved well by the proposed WBESPP scheme.

\begin{figure}[htbp!]
	\centering
    \subfigure[Overview.]{
		\includegraphics[width=0.85\linewidth]{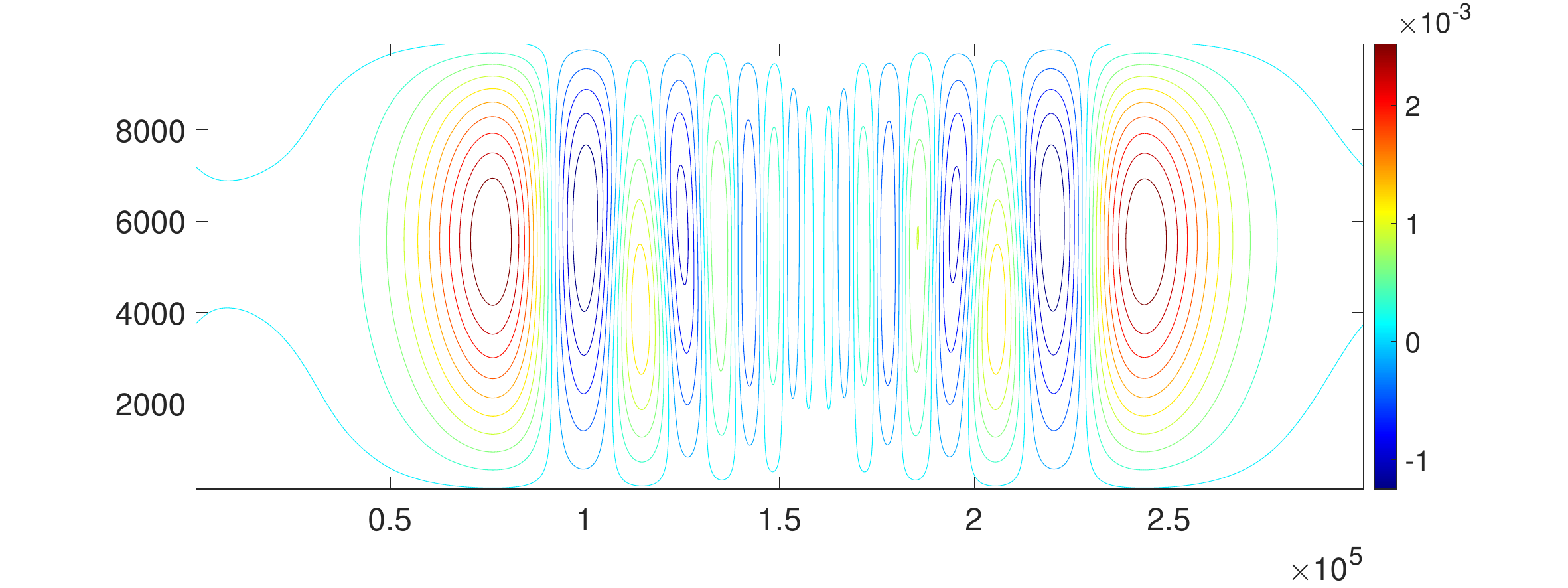}}
	\subfigure[Cut at $y=5000$.]{
		\includegraphics[width=0.85\linewidth]{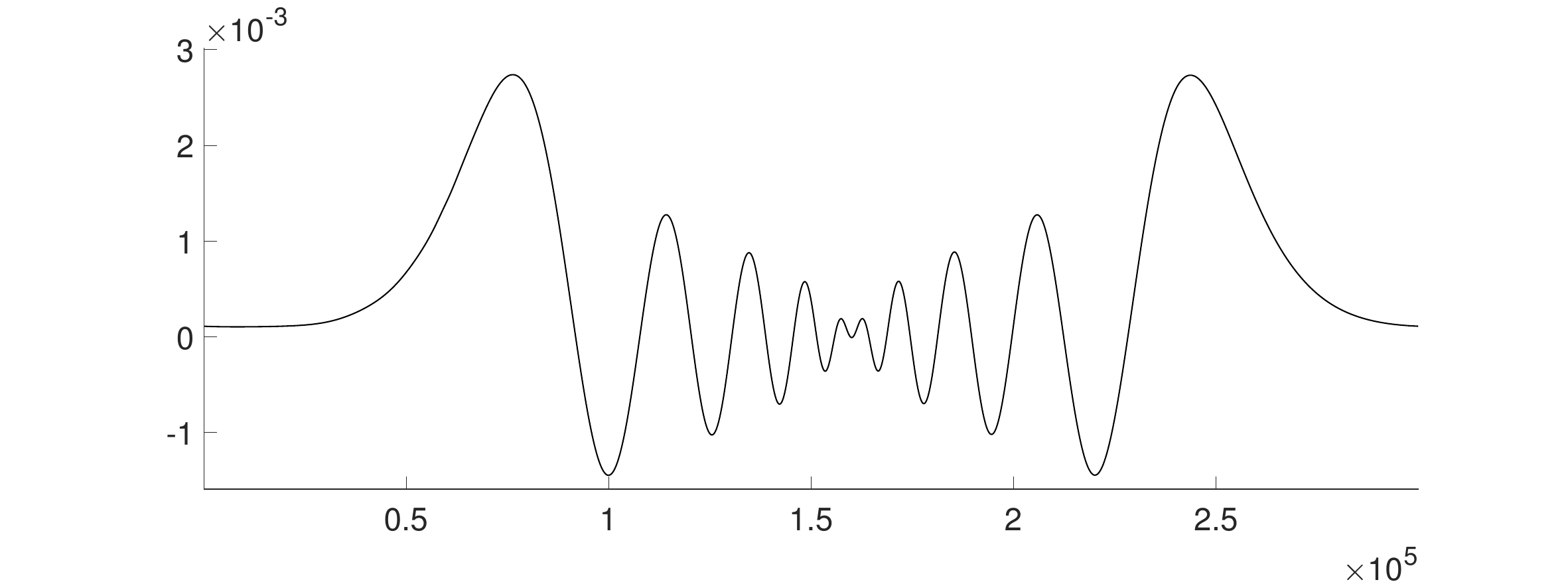}}
	\caption{Example 5.8: Inertia-gravity wave problem. The potential  temperature perturbation at $T=3000\mathrm s$ on $N_x\times N_y=1200\times 40$ meshes. 10 contour lines are used.}
   \label{figIG}
\end{figure}

\section{Concluding remarks}\label{sec6}

In this paper, we propose a structure preserving nodal DG scheme for solving the Euler equations with gravity which is well-balanced, entropy stable, and positivity-preserving. 
By rewriting the balance law and modifying the discretization of source term using an entropy conservative flux, we are able to achieve both the well-balancedness and entropy stability simultaneously. 
With a positivity preserving scaling limiter, we demonstrate that the fully-discrete scheme coupled with the forward Euler scheme ensures the positivity of density and pressure. Moreover, the limiter does not destroy the entropy stability and well-balancedness. 
 Theoretical analysis and numerical results confirm that our scheme maintains the numerical equilibrium state to machine accuracy, dissipates entropy, and preserves positive density and pressure even in extreme cases. 
Future works include the extension to unstructured meshes, and the application of other equations.

\bibliographystyle{abbrv}
\bibliography{tex}
\end{document}